\documentclass[a4paper,11pt]{article}

\usepackage{newtxtext,newtxmath}

\usepackage[active]{srcltx}
\usepackage[utf8]{inputenc}
\usepackage[english]{babel}
\usepackage{graphicx}
\usepackage[T1]{fontenc}
\usepackage{amsmath}

\usepackage{amsthm}
\usepackage{bm} 
\usepackage{bbm}
\usepackage{eqnarray}
\usepackage{tikz}
\usepackage[numbers]{natbib}
\usepackage{array}
\usepackage{esint}
\usepackage{enumitem}
\usepackage{makecell}
\usepackage[left=2.05cm,right=2.05cm,top=2.95cm,bottom=2.8cm]{geometry}

\usepackage[pdfborder={0 0 0}]{hyperref}

\numberwithin{equation}{section} 

\newcommand{\R}{\mathbb{R}}

\newcommand{\N}{\mathbb{N}}
\newcommand{\Prob}{\mathcal{P}}
\newcommand{\M}{\mathcal{M}} 
\newcommand{\A}{\mathcal{A}} 
\newcommand{\B}{\mathcal{B}} 
\newcommand{\F}{\mathcal{F}} 
\newcommand{\T}{\mathcal{T}}

\newcommand{\1}{\mathbbm{1}} 

\newcommand{\ddr}{\mathrm{d}} 
\newcommand{\dr}{\partial}
\newcommand{\Id}{\mathrm{Id}}

\newcommand{\mubf}{\bm{\mu}}
\newcommand{\nubf}{\bm{\nu}}

\newcommand{\Mm}{\mathcal{M}_{\pi_1 \to m}}

\newcommand{\TEE}{\mathcal{T}_{E, \mathrm{Eul}}}

\newcommand{\dst}[1]{\displaystyle{#1}}

\newtheorem{theo}{Theorem}[section]
\newtheorem*{theo*}{Theorem}
\newtheorem{prop}[theo]{Proposition}
\newtheorem{crl}[theo]{Corollary}
\newtheorem{lm}[theo]{Lemma}
\newtheorem{defi}[theo]{Definition}

\newtheorem{asmp}{Assumption}
\newtheorem*{asmp*}{Assumption}

\newtheorem*{question}{Question}

\theoremstyle{remark}
\newtheorem{rmk}[theo]{Remark}
\newtheorem{example}[theo]{Example}

\begin{document}

\title{Lifting functionals defined on maps to measure-valued maps \\ via optimal transport} 
\author{Hugo Lavenant}
\date{\today}

\maketitle

\begin{abstract}
How can one lift a functional defined on maps from a space $X$ to a space $Y$ into a functional defined on maps from $X$ into $\Prob(Y)$ the space of probability distributions over $Y$? Looking at measure-valued maps can be interpreted as knowing a classical map with uncertainty, and from an optimization point of view the main gain is the convexification of $Y$ into $\Prob(Y)$. We explain why trying to single out the largest \emph{convex} lifting amounts to solve an optimal transport problem with an infinity of marginals which can be interesting by itself. Moreover we show that, to recover previously proposed liftings for functionals depending on the Jacobian of the map, one needs to add a restriction of additivity to the lifted functional. 

\medskip

\end{abstract}

\section{Introduction}

We take for given $E$ a functional defined on (a subset of) maps $u : X \to Y$. As a running example, the reader can think that $X$, $Y$ are Euclidean spaces and 
\begin{equation}
\label{eq:introp_dirichlet}
E(u) = \frac{1}{2} \int_X |\nabla u|^2
\end{equation}  
is the Dirichlet energy, or more generally a functional of the type
\begin{equation}
\label{eq:intro_energy_generic}
E(u) = \int_X W(\nabla u(x)) \, \ddr x + \int_X f(x,u(x)) \, \ddr x.
\end{equation}
Here $\nabla u$ denotes the Jacobian matrix of $u$, and $W$ is convex. We consider measure-valued maps: instead of $u$ we take $\mubf : X \to \Prob(Y)$ valued in the space of probability distributions over $Y$. This can be interpreted as knowing $u(x) \in Y$ only up to uncertainty, thus replacing $u(x)$ by a probability distribution over $Y$. A classical map is recovered if we set $\mubf(x) = \delta_{u(x)}$ for any $x$, being $\delta_{u(x)}$ the Dirac mass located at $u(x)$. From a mathematical point of view the major gain is that the space $\Prob(Y)$ is convex, whatever structure one has on $Y$. In this framework, we study the following question.

\begin{question}[raw]
How can one lift the functional $E$ defined on maps $X \to Y$ into a functional $\T$ defined on the space of measure-valued maps $X \to \Prob(Y)$? 
\end{question}

This question is not new, in particular in the imaging community where functionals like~\eqref{eq:intro_energy_generic} are of great importance but typically suffer from non-convexity. In such a context, there have been several proposals for lifting of~\eqref{eq:intro_energy_generic} to measure-valued maps with a particular emphasis on the case of total variation, that is, when $W$ is a norm~\cite{laude2016sublabel, Vogt2018, vogt2020connection}. Non-convexity either comes from the term $\int_X f(x,u(x)) \, \ddr x$, usually referred as ``data-fitting term'', or from the space itself if $Y$ is a manifold: relaxing from $Y$ to $\Prob(Y)$ is really about gaining convexity.
Another proposal, equivalent to~\cite{vogt2020connection}, is to focus on the Dirichlet energy~\eqref{eq:introp_dirichlet} and to extend it to a Dirichlet energy for measure-valued maps. This has been suggested by Brenier~\cite[Section 3]{Brenier2003} and we worked on it in~\cite{Lavenant2019, LavenantPhD}. The motivation was to define \emph{harmonic maps} valued in the Wasserstein space, in analogy with the concept of geodesic valued in such space. Some independent work motivated by applications in surface matching~\cite{Solomon2013} proposed a similar definition of Dirichlet energy with the purpose of gaining convexity when $X$ and $Y$ are surfaces. Most proposed answers concentrate on functionals $E$ containing only first order derivatives, though there has been recently some proposal for functionals depending on the second order derivatives of the function $u$~\cite{vogt2019functional,loewenhauser2018functional,vogt2020connection}. Even if it is not the concern of this article, most of the aforementioned works also include a discretization and an algorithm to find minimizers of the (discretized) lifted functional.

\paragraph{Refining our investigation}

In the present work rather than proposing a new answer we aim at designing a framework to characterize the ``best'' possible relaxation. The starting point of our investigation is the following remark of Savaré and Sodini \cite{SavareSodoni2022} about the optimal transport problem, somewhat elementary, but very insightful. Given a function $c : Y \times Y \to [0, + \infty]$ of two variables, one can try to lift it up into a function $\T : \Prob(Y) \times \Prob(Y) \to [0, + \infty]$ defined over pair of measures. Following the canonical embedding $y \in Y \mapsto \delta_y \in \Prob(Y)$, we can lift $c$ into a map $\widetilde{\T}_c$ defined over $\Prob(Y)^2$ by 
\begin{equation*}
\widetilde{\T}_c(\mu,\nu) = \begin{cases}
c(y_1,y_2) & \text{if } \mu = \delta_{y_1} \text{ and } \nu = \delta_{y_2}, \\
+ \infty & \text{otherwise},
\end{cases}
\end{equation*}
but this extension is not convex nor lower semi-continuous for the topology of narrow convergence. To cope with these issues, it is natural to look for the convex and lower semi-continuous envelope, that is, the largest convex and lower semi-continous function below $\widetilde{\T}_c$. The work~\cite{SavareSodoni2022} characterizes this envelope, under suitable topological assumptions, as the optimal transport cost, defined for $\mu, \nu$ probability distributions over $Y$ by
\begin{equation*}
\mathcal{T}_c(\mu,\nu) = \min_\pi \left\{ \int_{Y \times Y} c(y_1,y_2) \, \pi(\ddr y_1, \ddr y_2) \ : \ \pi \in \Pi(\mu,\nu) \right\}.
\end{equation*}
Here $\Pi(\mu,\nu)$ is the set of probability distributions over $Y \times Y$ having marginals $\mu,\nu$. In~\cite{SavareSodoni2022} the authors use this observation to derive a very elegant proof of the Kantorovich duality. 

We can apply the same strategy in our case. Given a map $u : X \to Y$, we can embed it as a measure-valued map $\mubf_u : X \to \Prob(Y)$ via $\mubf_u(x) = \delta_{u(x)}$. Following the observation of Savaré and Sodini we investigate the following.

\begin{question}[refined]
Writing $\widetilde{\T}_E$ for the functional defined on the space of measure-valued maps $X \to \Prob(Y)$ by 
\begin{equation}
\label{eq:defi_tilde_T}
\widetilde{\T}_E(\mubf) = \begin{cases}
E(u) & \text{if } \mubf = \mubf_u \text{ for some } u : X \to Y, \\
+ \infty & \text{otherwise},
\end{cases}
\end{equation}
what is its convex and lower semi-continuous envelope, that is, the largest convex and lower semi-continuous functional $\T$ such that $\T \leq \widetilde{\T}$?
\end{question}

We need to specify the topology on the space of measure-valued map for this question to be well-defined. We will put the topology of narrow convergence on the measures over the product space $X \times Y$, as we detail below. We will show that the envelope is the ``Lagrangian'' lifting $\T_E$, which consists in solving a multimarginal optimal transport problem with an infinity of marginals. On the other hand, this lifting $\T_E$ does \emph{not} coincide with the liftings introduce in the imaging or optimal transport community~\cite{Vogt2018, vogt2019lifting, Brenier2003, Lavenant2019, Solomon2013}. These works rather rely on a ``Eulerian'' lifting introduced below, that we call $\TEE$. We will also give a characterization of the lifting $\TEE$, as a convex lower semi-continuous envelope, but with an additional requirement of subadditivity.

\paragraph{The largest relaxation}

We first concentrate on the convex and lower semi-continuous envelope of $\widetilde{\T}_E$, defined in~\eqref{eq:defi_tilde_T}.
In this setting $X,Y$ are Polish spaces. As in the case of the Dirichlet energy the maps $u$ are only defined up to a negligible set, we allow for a reference measure $m$ on $X$, and we work rather with $L^0(X,Y;m)$ or $L^0(X,\Prob(Y);m)$ the space of equivalence classes of maps from $X$ to $Y$ or to $\Prob(Y)$, quotiented by the relation of being equal $m$-a.e. Under a suitable assumption displaying an interplay between the functional $E$, the spaces $X,Y$ and the measure $m$ (see Assumption~\ref{asmp:reg_XYm_E}), we are able 
to prove that to the convex and lower semi-continuous envelope of $\widetilde{\T}_E$ is
\begin{equation}
\label{eq:intro_generalized_MMOT}
\mathcal{T}_E(\mubf) = \inf_Q \left\{ \int_{L^0(X,Y;m)} E(u) \, Q(\ddr u)  \ : \  Q \in \Pi(\mubf) \right\},
\end{equation}
where now $Q$ is now a probability measure on $L^0(X,Y;m)$ generalizing the transport plan $\pi$, and the assertion $Q \in \Pi(\mubf)$ generalizing the marginal condition reads, as an equality $m$-a.e. of $\Prob(Y)$-valued maps,
\begin{equation*}
\mubf = \int_{L^0(X,Y;m)} \mubf_u \, Q(\ddr u).
\end{equation*}

\begin{center}
\fbox{\begin{minipage}{0.9\textwidth}
The first main contribution of this article is to formalize properly the right hand side of~\eqref{eq:intro_generalized_MMOT} and prove (see Theorem~\ref{theo:multimarginalOT_lifting}) that it coincides with the convex and lower semi-continuous envelope of $\widetilde{\T}_E$ defined in~\eqref{eq:defi_tilde_T}. A corollary of this result is the dual formulation of $\T_E$ (see~\eqref{eq:duality_MMOT}).
\end{minipage}}
\end{center}

\paragraph{Reinterpreting previous results}

Problem~\eqref{eq:intro_generalized_MMOT} can actually be seen as a multimarginal optimal transport problem with an infinity of marginals. Specifically, when the space $X = \{ x_1, x_2, \ldots, x_N \}$ is finite and contains $N$ points, and $m$ is the counting measure, unpacking~\eqref{eq:intro_generalized_MMOT} reveals a multimarginal problem with $N$ marginals. In this case, the cost function $c : Y^N \to [0, + \infty]$ is recovered as
\begin{equation*}
c(u(x_1), u(x_2), \ldots, u(x_N)) = E(u),
\end{equation*}
and a map $\mubf$ from $X$ to $\Prob(Y)$ is naturally seen as a collection of $N$ marginals, that is, an element of $\Prob(Y)^N$.  

More interestingly, when $(X,m)$ is a segment of $\R$ endowed with the Lebesgue measure, say $X = [0,1]$, while $Y$ is a Banach space, and $E$ is an functional of the type
\begin{equation*}
E(u) = \frac{1}{p} \int_0^1 |\dot{u}_t|^p \, \ddr t
\end{equation*}
for some $p \in (1,+ \infty)$, then the problem~\eqref{eq:intro_generalized_MMOT} is the one which appears when studying the probabilistic representation of the continuity equation, also known as superposition principle. In this case $\mubf = (\mu_t)_{t \in [0,1]}$ is a curve valued in the space $\Prob(Y)$, and the seminal work of Lisini~\cite{Lisini2007} (see also~\cite[Chapter 8]{AGS}) yields alternative and meaningful expressions for the lifted functional $\T_E$:
\begin{align}
\label{eq:intro_dynamic_curves}
\mathcal{T}_E((\mu_t)_{t \in [0,1]}) & = \min_v \left\{ \frac{1}{p} \int_0^1 \int_Y |v(t,y)|^p \, \mu_t(\ddr y) \, \ddr t \ : \ \dr_t \mu + \mathrm{div}(v \mu) = 0 \right\} \\
\label{eq:intro_metric_curve}
& = \begin{cases}
\displaystyle{\frac{1}{p} \int_0^1 |\dot{\mu}_t|^p \, \ddr t} & \text{if } \mu \in \mathrm{AC}^p([0,1],(\Prob_p(Y), W_p)), \\
 + \infty & \text{otherwise}. 
\end{cases} 
\end{align}
In the first equation the minimum is taken over all $v = v(t,y)$ such that $\| v(t,\cdot) \|^p_{L^p(Y,\mu_t)}$ is integrable in time, and where the continuity equation $\dr_t \mu + \mathrm{div}(v \mu) = 0$ is understood in a weak sense. In this context, $(\mu_t)_{t \in [0,1]}$ describes the evolution in time of a distribution of mass, and $v(t,y)$ is interpreted as the velocity of the particle located at time $t$ in position $y$. The formula~\eqref{eq:intro_generalized_MMOT} for $\T_E$ corresponds to a \emph{Lagrangian} point of view: the measure $Q$ tracks the distributions of all trajectories $(u(t))_{t \in [0,1]}$. The right hand side of~\eqref{eq:intro_dynamic_curves} is usually called the \emph{Eulerian} point of view: one only records $v(t,y)$ the velocities at different instant and positions. If both $Q$ and $v$ are the optimal ones, we have precisely $\dot{u}(t) = v(t,u(t))$ for a.e. $t$ and for $Q$-a.e. curve $u$. In the second equation above, $\mathrm{AC}^p$ stands for ``$p$-absolutely continuous curves''~\cite[Chapter 1]{AGS}, and these curves are valued in $(\Prob_p(Y), W_p)$ the Wasserstein space of order $p$ built over $Y$~\cite[Chapter 7]{AGS}. The expression $|\dot{\mu}_t|$ stands for the metric derivative~\cite[Chapter 1]{AGS}.  

In the case $E$ is an functional of higher order, specifically if 
\begin{equation*}
E(u) =  \int_0^1 |\ddot{u}_t|^2 \, \ddr t,
\end{equation*}
then the functional $\T_E$ has been proposed as quantity to minimize in order to recover \emph{splines} valued in the Wasserstein space \cite{benamou2019second, chen2018measure}. 
In~\cite{benamou2019second}, at least three possible models for splines valued in the Wasserstein are investigated. The one corresponding to $\T_E$, with $E$ as above, is what the authors called ``the Kantorovich relaxation'' and it is proposed by analogy with relaxations proposed in fluid dynamics~\cite{brenier1989least}. Our result shed a new light on this ``Kantorovich relaxation'', by proving it is the convex and lower semi-continuous envelope of the functional for splines for classical curves.


\paragraph{Beyond the case of curves}

When $X = \Omega$ is a bounded subset of an Euclidean space $\R^d$, endowed with its Lebesgue measure $m$, while $Y$ is a Euclidean space or a Riemannian manifold, the situation becomes more intricate. Concentrating on the case
\begin{equation*}
E(u) =  \int_\Omega W(\nabla u(x)) \, \ddr x
\end{equation*}  
for a convex function $W : TY \otimes \R^d \to [0, + \infty]$ taking in argument the Jacobian matrix of $u$, several authors have followed the Eulerian point of view and have proposed a functional mimicking the dynamical formulation~\eqref{eq:intro_dynamic_curves} for curves of measures, which is different than $\T_E$ and which reads 
\begin{equation*}
\TEE(\mubf) =  \min_{v} \left\{ \int_{\Omega \times Y} W(v(x,y)) \,  \mubf_x(\ddr y) \, \ddr x \ : \ \nabla_x \mubf + \nabla_y \cdot (v \mubf) = 0 \right\},
\end{equation*}
where the minimum is taken over $v : \Omega \times Y \to TY \otimes \R^d$ which can be interpreted as a ``field of Jacobian matrix'', that is, $v(x,y)$ represents a linear operator from $\R^d = T_x \Omega$ into $T_yY$. The equation $\nabla_x \mubf + \nabla_y \cdot (v \mubf) = 0$ which has to be interpreted in a weak sense is the natural extension of the continuity equation $\dr_t \mu + \nabla \cdot(v \mu)$ featured in~\eqref{eq:intro_dynamic_curves}. We refer to the works~\cite{Brenier2003,Solomon2013,Lavenant2019,vogt2020connection} already mentioned and to Section~\ref{sec:lifting_TEE} of the present work to see $\TEE$ presented in this way. It also has a ``dual'' formulation which reads
\begin{multline*}
\TEE(\mubf) =  \sup_{\varphi} \Bigg\{ - \int_{\Omega \times Y} \left( \nabla_x \cdot \varphi(x,y) + W^*(\nabla_y \varphi(x,y))  \right) \, \mubf_x(\ddr y) \, \ddr x \ : \ \varphi \in C^1_c(\Omega \times Y, \R^d) \\
 \text{ and }W^*(\nabla_y \varphi) < + \infty \text{ on } \Omega \times Y  \Bigg\},
\end{multline*}
where now the supremum is taken over smooth and compactly supported test functions. The works~\cite{laude2016sublabel, Vogt2018, vogt2019lifting} introduced $\TEE$ via this dual formulation. 

In the present work we focus on $Y$ being a Euclidean space instead of a Riemannian manifold, as it is challenging enough to analyze the difference between $\T_E$ and $\TEE$ in this setting.

\paragraph{Understanding the difference between $\T_E$ and $\TEE$}

Though the Eulerian lifting is convex, lower semi-continuous, and below $\widetilde{\T}_E$, and even though $\TEE = \T_E$ in the case $X = [0,1]$ as stated in~\eqref{eq:intro_dynamic_curves}, in general one has $\TEE(\mubf) < \T_E(\mubf)$.

Our second contribution is to explain precisely why this discrepancy happens, and to characterize $\TEE$ as a convex and lower semi-continuous envelope of $\widetilde{\T}_E$, but over a restricted class of measure-valued functional. 
The problem has to do with the integral representation of $\TEE$, which implies locality and additivity, while $\T_E$ does not satisfy it. Precisely, following the literature on $\Gamma$-convergence~\cite[Chapter 14 and 15]{dalMaso2012introduction}, we see our functionals as depending on both a function $u$ and an open set $A \subseteq \Omega$, that is 
\begin{equation*}
E(u,A) =  \int_A W(\nabla u(x)) \, \ddr x.
\end{equation*} 
Such a functional is local (the value of $E(u,A)$ only depends on the restriction of $u$ to $A$) and moreover \emph{additive}, in the sense that
\begin{equation*}
E(u, A_1 \cup A_2) = E(u, A_1) + E(u,A_2)
\end{equation*}
provided $A_1$, $A_2$ are disjoint sets. By construction $\TEE$ is also additive, so that for disjoint sets $A_1$, $A_2$,
\begin{equation*}
\TEE(\mubf, A_1 \cup A_2) = \TEE(\mubf, A_1) + \TEE(\mubf,A_2).
\end{equation*}
On the other hand, the example of the theory of $Q$-functions~\cite{DeLellisSpadaro} already mentioned in one of our previous work~\cite{Lavenant2019} shows that $\T_E$ is not additive, but only superadditive: $\T_E(u, A_1 \cup A_2) \geq \T_E(u, A_1) + \T_E(u,A_2)$, and the inequality is strict in general. Actually, this is the case when $X = \mathbb{S}^1$ is the unit circle.

\begin{center}
\fbox{\begin{minipage}{0.9\textwidth}
Our second main contribution, in Theorem~\ref{theo:main}, is to prove that the Eulerian lifting $\TEE$ is the convex, lower semi-continuous and \emph{subadditive} envelope of the functional $\widetilde{\T}_E$ defined in~\eqref{eq:defi_tilde_T}.
\end{minipage}}
\end{center}

\paragraph{An open question}

Consider a functional of second order for a map $u : \Omega \subset \R^d \to \R^q$:
\begin{equation*}
E(u) = \int_\Omega H(\nabla^2 u(x)) \, \ddr x,
\end{equation*}
being $H$ a convex function over $q \times d \times d$ tensors.
What is the convex, lower semi-continuous and \emph{subadditive} envelope of the functional $\widetilde{\T}_E$ in this case?
Though some lifting have been proposed \cite{loewenhauser2018functional,vogt2019functional}, it is not clear that they are optimal ones, and their theoretical properties are not well understood. We leave this investigation for future work.

\paragraph{Other possible lifting spaces} 

In this work we restrict the study to the convexification of $Y$ into $\Prob(Y)$. There are other fruitful ways to embed the space of functions from $X$ to $Y$ into a larger space to gain convexity. If $Y = \R$ is one dimensional, then one could replace $u$ by its sublevel set function $(x,z) \mapsto \1_{\{ z \leq u(x) \}}$ and optimize in this larger space of hypographs \cite{pock2008convex, pock2010global}. This method, which originated from calculus of variations~\cite{alberti2003calibration} has been extended also to the vector-valued case~\cite{strekalovskiy2014convex}. We refer the reader to Section 2.6.1 of~\cite{vogt2020connection} for some connections between this approach and the one of the present work. More recently the authors of~\cite{mollenhoff2019lifting} proposed to embed $u$ as a current on the space $X \times Y$ (concentrated on the graph of $u$) in order to obtain a convex relaxation of~\eqref{eq:intro_energy_generic} when $W$ is not convex, but only polyconvex. As a side remark, for the lifting of a subclass of polyconvex functionals, we also proposed in~\cite{ghoussoub2021hidden} a formulation for measure-valued maps based on a generalization of $\TEE$, but such generalization is probably not additive and does not fit well in the framework of this article.

\paragraph{Organization of the paper}

After recalling some notations and classical results in Section~\ref{sec:preliminaries}, we state the multimarginal optimal transport problem with an infinity of marginals in Section~\ref{sec:generalMMOT},
and show it is the convex and lower semi-continuous envelope of $\widetilde{\T}_E$. We then move to the Eulerian lifting and its main properties in Section~\ref{sec:lifting_TEE}. Eventually we explain the difference between the two liftings: that $\T_E$ is in general not additive, while $\TEE$ is the convex, lower semi-continuous and subadditive envelope of $\widetilde{\T}_E$. 
This is done in Section~\ref{sec:difference_liftings}. We collect some technical results which would have overburden the core of the article in an appendix.  

\section*{Acknowledgments}

We would like to thank Giuseppe Savaré for several interesting discussions and for introducing us to the literature on the localization of functionals presented in Section~\ref{sec:localization}, Simone Di Marino for pointing out an error in an earlier version of this article, as well as an anonymous referee for suggestions improving the clarity of this work. The author is affiliated to the Bocconi Institute for Data Science and Analytics (BIDSA). 

\section{Running assumptions and some preliminaries}
\label{sec:preliminaries}

We first settle some notations, recall standard results, and introduce measure-valued maps. 

Note that all functionals that we consider will take values in $[0, + \infty]$: they are non-negative (the extension to functionals bounded from below by a constant is straightforward) and they may take the value $+ \infty$. We use the abbreviation ``l.s.c.'' for ``lower semi-continuous''. 

\paragraph{Some standard notations}

We will use $a \cdot b$ to denote either the canonical scalar product between vectors $a,b \in \R^d$, or the canonical scalar product between matrices $a,b \in \R^{q \times d}$, given some integers $d, q \geq 1$. On the other hand, the product $ab$ will stand for the matrix vector product between $a \in \R^{q \times d}$ and $b \in \R^d$. 

For two metric spaces $X$ and $Y$ we denote by $C_b(X,Y)$ the space of continuous and bounded functions from $X$ to $Y$. If $Y$ is a vector space then $C_c(X,Y)$ is the space of continuous and compactly supported functions from $X$ to $Y$. If $Y = \R$ then we omit it and only write $C_b(X)$ and $C_c(X)$.  

\paragraph{Measure theory on Polish spaces}

On the questions of measure theory we refer the reader to the monographs \cite[Chapters 1 and 2]{AFP} and \cite[Chapter 5]{AGS}. 

We always take $X,Y$ Polish spaces, that is, topological separable spaces whose topology is generated by a complete metric. We endow them with their Borel $\sigma$-algebra. For a Polish space $X$ we denote by $\M(X)$ the space of finite Borel measures on it, by $\M_+(X)$ the space of finite non-negative measures, and by $\Prob(X)$ the space of probability measures. The space $\M(X)^{qd}$ for integers $q,d \geq 1$ will stand for the space of $q \times d$ matrix-valued measures. For $\sigma \in \M(X)$ or in $\M(X)^{qd}$, we denote by $|\sigma| \in \M_+(X)$ its total variation measure~\cite[Definition 1.4]{AFP}.


A sequence $(\sigma_n)_{n \in \N}$ of measures over $X$ is said to converge narrowly to $\sigma$ if for every continuous and bounded function $\varphi \in C_b(X)$, there holds
\begin{equation*}
\lim_{n \to + \infty} \int_X \varphi(x) \, \sigma_n(\ddr x) = \int_X \varphi(x) \, \sigma(\ddr x).  
\end{equation*}
Endowed with the topology of narrow convergence, the space $\M_+(X)$ is metrizable~\cite[Theorem 8.3.2]{bogachev}, in particular the topology is characterized by converging sequences.
A property that we will use repeatedly is the following, see~\cite[Equation (5.1.15)]{AGS} for the proof.  

\begin{lm}
\label{lm:lsc_implies_lsc_int}
If $\varphi : X \to [0, + \infty]$ is l.s.c. and non-negative then the function $\sigma \mapsto \int_X \varphi(x) \, \sigma(\ddr x)$
is l.s.c. for the topology of narrow convergence on $\M_+(X)$.
\end{lm}

A subset $\mathcal{K}$ of $\Prob(X)$ is said \emph{tight} if for every $\varepsilon > 0$ we can find a compact set $K \subseteq Y$ such that $\mu(K) \geq 1 - \varepsilon$ for any $\mu \in \mathcal{K}$. 
Prokhorov's theorem states that a subset $\mathcal{K}$ of $\Prob(X)$ is tight if and only if it is relatively compact for the topology of narrow convergence. 

On the other hand, on the set $\M(X)$ we can also consider weak-$\star$ converging sequences.
A sequence $(\sigma_n)_{n \in \N}$ converges weak-$\star$ if for every continuous and compactly supported function $\varphi \in C_c(X)$, there holds
\begin{equation*}
\lim_{n \to + \infty} \int_X \varphi(x) \, \sigma_n(\ddr x) = \int_X \varphi(x) \, \sigma(\ddr x).  
\end{equation*}
Clearly, weak-$\star$ convergence is implied by narrow convergence, but the converse does not hod except if $X$ is compact. For $\mathcal{K}$ a subset of $\M(X)$ to be sequentially relatively compact for weak-$\star$ convergence, it is enough to have $|\sigma|(X)$ uniformly bounded for $\sigma \in \mathcal{K}$ \cite[Theorem 1.59]{AFP}. 

Given a measure $\mu$ on $X$ and a measurable map $T : X \to Y$, we call $T \# \mu$ the image measure of $\mu$ by $T$, defined by $(T \# \mu)(B) = \mu(T^{-1}(B))$ for any Borel set $B \subseteq Y$. 

\paragraph{The space $L^0(X,Y;m)$}

For $X,Y$ Polish spaces and $m \in \M_+(X)$ a finite non-negative measure on $X$, we define $L^0(X,Y;m)$ as the space of measurable maps from $X$ to $Y$ quotiented by the equivalence relation of coinciding $m$-a.e. We put on it the distance $d_{L^0(X,Y;m)} (u,v) = \int_X \min ( d_Y(u(x), v(x)), 1) \, m (\ddr x)$
being $d_Y$ a complete distance metrizing the topology of $Y$. Convergence of a sequence for the $d_{L^0(X,Y;m)}$ distance is equivalent to requiring that any subsequence has a subsequence which converges pointwise $m$-a.e. The space $L^0(X,Y;m)$ is Polish.

\paragraph{Measure-valued maps and disintegration theorem}

We fix $m \in \M_+(X)$ a non-negative finite measure on the space $X$. We consider measurable maps $\mubf$ from $X$ to $\Prob(Y)$, but most of the time we will see them as measures on the product space $X \times Y$. 

\begin{defi}
A measure $\mubf \in \M_+(X \times Y)$ is said to have first marginal $m$ if $\pi_1 \# \mubf = m$, being $\pi_1 : X \times Y \to X$ the projection on the first coordinate. We denote by $\Mm(X \times Y)$ the set of non-negative measures on $X \times Y$ whose first marginal is $m$ and we endow it with the topology of narrow convergence.  
\end{defi} 

The set $\Mm(X \times Y)$ is easily seen to be closed for the topology of weak convergence. The identification between measure-valued maps $X \to \Prob(Y)$ and measures on the product space $X \times Y$ comes from the disintegration theorem \cite[Theorem 2.28]{AFP}. 

\begin{theo}
\label{theo:disintegration}
A positive measure $\mubf$ on $X \times Y$ belongs to $\Mm(X \times Y)$ if and only if there exists a measurable family $(\mubf_x)_{x \in X}$ of probability distributions on $Y$ such that, for every measurable and bounded function $\varphi : X \times Y \to \R$, 
\begin{equation*}
\int_{X \times Y} \varphi(x,y) \, \mubf(\ddr x, \ddr y) = \int_X \left( \int_Y \varphi(x,y) \, \mubf_x(\ddr y) \right) \, m(\ddr x).
\end{equation*}
In such a case, the family $(\mubf_x)_{x \in X}$ is uniquely determined by $\mubf$ up to a $m$-negligible set. 
\end{theo} 

\noindent We we call $(\mubf_x)_{x \in X}$ the disintegration of $\mubf$ with respect to the first coordinate. Theorem~\ref{theo:disintegration} can be reformulated as: the set $\Mm(X \times Y)$ coincides as a set (but not topologically) with the space $L^0(X,\Prob(Y);m)$.

\paragraph{Elements of convex analysis} 

We recall here some results of convex analysis which can be found in~\cite{EkelandTeman}. We consider the general case where we have two vector spaces $V,V'$ paired by a non-degenerate bilinear form $\langle \cdot, \cdot \rangle$. The space $V$ is endowed with the topology $\sigma(V,V')$, that is, the coarsest topology making the maps $v \in V \mapsto \langle v,w \rangle$ continuous for any fixed $w \in V'$. As an example, if $V$ is the space of finite signed measures while $V'$ is the space of continuous and bounded functions (resp. continuous and compactly supported functions), then $\sigma(V,V')$ coincides with the topology of narrow convergence (resp. weak-$\star$ convergence).

Given a function $F : V \to [0, + \infty]$ which is not identically $+ \infty$, we define its Legendre transform $F^*$ over $V'$ by, for $w \in V'$,
\begin{equation*}
F^* (w) = \sup_{v \in V} \; \langle v,w \rangle - F(v).  
\end{equation*} 
The Fenchel-Moreau duality \cite[Propositions 3.1 and 4.1]{EkelandTeman} implies that, by defining 
\begin{equation*}
F^{**}(v) = \sup_{w \in V'} \; \langle v,w \rangle - F^*(w),
\end{equation*}
then $F^{**}$ is the convex and l.s.c. envelope of $F$ (for the $\sigma(V,V')$ topology). In particular, if $F$ is convex, l.s.c. and not identically $+ \infty$, then as functions over $V$
\begin{equation}
\label{eq:FenchelMoreau}
F^{**} = F. 
\end{equation}

Eventually, as it will be useful, we call \emph{convex indicator} of a set $A \subseteq V$ the function defined over $V$ which takes the value $0$ on $A$ and $+ \infty$ outside of $A$. The usual indicator of the set $A$ taking the value $1$ on the set and $0$ otherwise is $\1_A$.   

\paragraph{Notations for Sobolev spaces}

As it will not be the focus of this article, we only recall the notations we use and point to textbooks for more details. When $\Omega$ is an open subset of $\R^d$ endowed with its Lebesgue measure $m$ while $k \geq 1$ is an integer and $p \in (1,+\infty)$, we denote by $W^{k,p}(\Omega, \R^q) \subseteq L^0(\Omega, \R^q;m)$ the set of maps having all derivatives up to order $k$ in $L^p$, see for instance~\cite{adams2003sobolev}. In the case $k=1$ and $p=1$, we rather consider the space $\mathrm{BV}(\Omega, \R^q)$ of functions of bounded variations, that is, whose first derivative is a finite (vector-valued) measure, see~\cite[Chapter 3]{AFP}.

For a map $u : \Omega \subseteq \R^d \to \R^q$, we denote by $\nabla u$ its Jacobian matrix, when it exists. It is a $q \times d$ matrix-valued field given as $(\partial_{x_j} u_i)_{1 \leq i \leq q, \, 1 \leq j \leq d}$ representing the differential of $u$ in the canonical basis of $\R^q$ and $\R^d$.

\section{The Lagrangian lifting, or the multimarginal optimal transport problem with an infinity of marginals}
\label{sec:generalMMOT}

We fix $Y$ a Polish space, which would be the one over which ``mass'' is transported. The space $X$ is the one indexing the marginals. It is also a Polish space, and is endowed with a finite measure $m \in \M_+(X)$.

As described above, the set of measure-valued maps is identified with $\Mm(X \times Y)$, and is endowed with the topology of narrow convergence. If $u \in L^0(X,Y;m)$, we define $\mubf_u \in \Mm(X \times Y)$ as the measure on $X \times Y$ defined as $(\Id,u) \# m$. It has first marginal $m$ and is characterized by: for every continuous and bounded function $\varphi : X \times Y \to \R$, 
\begin{equation*}
\int_{X \times Y} \varphi(x,y) \, \mubf_u(\ddr x, \ddr y) = \int_X \varphi(x,u(x)) \, m(\ddr x).
\end{equation*}
The disintegration of the measure $\mubf_u$ with respect to its first coordinate is given by $(\delta_{u(x)})_{x \in X}$: the notation $\mubf_u$ is consistent with the one we used in the introduction.

Given $E : L^0(X,Y;m) \to [0, + \infty]$, we define $\widetilde{\T}_E$ by the formula~\eqref{eq:defi_tilde_T}, that is, $\widetilde{\T}_E(\mubf_u) = E(u)$ for $u \in L^0(X,Y;m)$ and $+ \infty$ otherwise. In this section, we want to characterize the convex and l.s.c. envelope of $\widetilde{\T}_E$, that is, the largest convex and l.s.c. functional on $\Mm(X \times Y)$ below $\widetilde{\T}_E$.

\subsection{The case of functionals of order zero}

We start with a case where our whole new framework is too heavy: the case of functionals of order zero, that is, of functionals $E$ which can be written $E(u) = \int_{X \times Y} f(x,u(x)) \, m(\ddr x)$ for some non-negative and l.s.c. function $f$. Indeed, in this case, the lifting is not only convex, but linear. It is reminiscent of the classical lifting in optimalization which replaces the minimization of a real-valued function $g$ over $Y$ by the minimization of $\int_Y g(y) \, \mu(\ddr y)$ over all probability distributions $\mu$ over $Y$ in order to convexify the problem. In our case, the result reads as follows.    

\begin{prop}
\label{prop:envelope_linear}
Let $f: X \times Y \to [0, + \infty]$ a non-negative l.s.c. function. We define $E$ on $L^0(X,Y;m)$ by
\begin{equation*}
E(u) = \int_X f(x,u(x)) \, m(\ddr x).
\end{equation*}
Then the functional $\T_E$ defined on $\Mm(X \times Y)$ by
\begin{equation*}
\T_E(\mubf) = \int_{X \times Y} f(x,y) \, \mubf(\ddr x, \ddr y),
\end{equation*}
is the convex and l.s.c. envelope of $\widetilde{\T}_E$. It satisfies in addition $\T_E(\mubf_u) = \widetilde{\T}_E(\mubf_u) = E(u)$ for all $u \in L^0(X,Y;m)$.
\end{prop}

\begin{proof}
It is straightforward to prove that $\T_E$ is linear and l.s.c. (see Lemma~\ref{lm:lsc_implies_lsc_int}), and moreover that $\T_E(\mubf_u) = \widetilde{\T}_E(\mubf_u) = E(u)$ for all $u \in L^0(X,Y;m)$.

The proof of the maximality of $\T_E$ is more delicate and we rely on results proved below. We first claim that $\T_E$, introduced in Proposition~\ref{prop:envelope_linear} coincides with $\T_E$ introduced in~\eqref{def:general_OT} thanks to Proposition~\ref{prop:Pimu_non_empty}. Thus Proposition~\ref{prop:TE_upper_bound} shows that any $\T : \Mm(X \times Y) \to [0,+\infty]$ which is convex, l.s.c. and such that $\T \leq \widetilde{\T}_E$ satisfies $\T \leq \T_E$. 
\end{proof} 

\subsection{The multimarginal optimal transport problem with an infinity of marginals}

We now move on to the general case: we first define the multimarginal optimal transport problem with an infinity of marginals, and then show that it characterizes the convex and l.s.c. envelope.
As a safety check in this section, the reader can always take $X$ a finite set endowed with $m$ the counting measure, which corresponds ultimately to the case of multimarginal optimal transport with a finite number of marginals indexed by $X$, see Section~\ref{sec:MMOT} for more details.

We will consider transport plans $Q$ as probability measures on the space $L^0(X,Y;m)$. The first step is to understand what the marginals of a probability measure $Q$ are. 

%
%

\begin{defi}
\label{def:marginal_Q}
Let $Q \in \Prob(L^0(X,Y;m))$ a probability measure on functions, and $\mubf \in \Mm(X \times Y)$ a measure on the product space $X \times Y$ whose first marginal is $m$. We say that $Q \in \Pi(\mubf)$, to be read ``$Q$ has marginals $\mubf$'' if: for every continuous and bounded function $\varphi : X \times Y \to \R$, there holds
\begin{equation*}
\int_{X \times Y} \varphi(x,y) \, \mubf(\ddr x, \ddr y) =  \int_{L^0(X,Y;m)} \left( \int_X \varphi(x,u(x)) \, m(\ddr x) \right) Q(\ddr u).
\end{equation*}
\end{defi}

\noindent Standard arguments yield that the identity above can be extended to any measurable and bounded function $\varphi : X \times Y \to \R$. Note that the definition can be read as
\begin{equation}
\label{eq:def_marginal_Q}
\mubf = \int_{L^0(X,Y;m)} \mubf_u \, Q(\ddr u),  
\end{equation}
where the equality holds as elements of $\Mm(X \times Y)$, and the integral in the right hand side is interpreted a the intensity measure~\cite[Section 2.1]{kallenberg2017random} of the random measure $\mubf_u$ when $u$ is a random map with law $Q$.

\begin{example}
\label{ex:parametric_Q}
To give an intuition to the reader, and as we will use it throughout the course of this work, let us expand what happens if we try to write $Q$ in ``parametric'' form to reduce complexity. 

Consider a ``latent'' measurable space $Z$, and assume that we have a function $\Psi : X \times Z \to Y$ jointly measurable. For any $z \in Z$, we can look at $\Psi(\cdot, z) = u_z$: it yields a measurable function from $X$ to $Y$. The map $z \mapsto u_z \in L^0(X,Y;m)$ is measurable thanks to Fubini's theorem. Thus, for $\lambda \in \Prob(Z)$, we can define $Q = (z \mapsto u_z) \# \lambda$ as a probability measure on $L^0(X, Y;m)$. From a probabilistic point of view, $Q$ denotes the law of $u_z = \Psi(\cdot,z)$ if the ``latent variable'' $z$ follows the law $\lambda$. 

For $x \in X$, let $\mubf^Q_x = \Psi(x, \cdot) \# \lambda \in \Prob(Y)$. We claim that $Q \in \Pi(\mubf^Q)$ where $\mubf^Q \in \Mm(X \times Y)$ has disintegration with respect to the first coordinate given by $(\mubf_x^Q)_{x \in X}$. Indeed, for a continuous and bounded function $\varphi : X \times Y \to \R$ we have
\begin{align*}
\int_{L^0(X,Y;m)} \left( \int_X \varphi(x,u(x)) \, m(\ddr x) \right) \, Q(\ddr u) 
& = \int_Z \left( \int_X \varphi(x,\Psi(x,z)) \, m(\ddr x) \right) \, \lambda(\ddr z) \\
& = \int_{X} \left( \int_Z \varphi(x,\Psi(x,z)) \, \lambda(\ddr z) \right) \, m(\ddr x) \\
& = \int_X \left( \int_Y \varphi(x,y) \, \mubf^Q_x(\ddr y) \right) \, m(\ddr x),
\end{align*}
where we have used successively the definition of $Q$, Fubini's theorem, and the definition of $\mubf^Q_x$. In the last term we recognize the measure whose disintegration with respect to the first coordinate is $m$-a.e. equal to $(\mubf_x^Q)_{x \in X}$, see Theorem~\ref{theo:disintegration}. 
\end{example}

\begin{rmk}
\label{rmk:continuous_marginals}
An interesting case of the example above is when we take for $Z$ a set of functions. For instance if $Q$ is concentrated on $C(X,Y)$ the space of continuous functions then we can take $Z = C(X,Y)$ and $\Psi(x,u) = u(x)$, and put on $Z$ the measure $Q$ itself. With $e_x : u \mapsto u(x) \in Y$ which maps a function $u \in C(X,Y)$ to its evaluation at the point $x$, we have $\mubf^Q_x = e_x \# Q \in \Prob(Y)$. Thus for a $\mubf \in \Mm(X \times Y)$ with disintegration with respect to the first coordinate given by $(\mubf_x)_{x \in X}$, Example~\ref{ex:parametric_Q} yields
\begin{equation*}
Q \in \Pi(\mubf) \qquad \Leftrightarrow \qquad \text{for } m \text{-a.e. } x \in X, \; \mubf^Q_x = \mubf_x.
\end{equation*}
That is, the marginals of the measure $Q$ are given by the collection $\mubf^Q_x = e_x \# Q$ for $x \in X$.

The problem to extend this computation to an arbitrary $Q \in \Prob(L^0(X,Y;m))$ relies on the evaluation operator: indeed as the elements of $L^0(X,Y;m)$ are only defined $m$-a.e., $e_x : u \mapsto u(x)$ is a priori not well-defined, nor measurable, on $L^0(X,Y;m)$. That is why in this latter case we resort to Definition~\ref{def:marginal_Q}.
\end{rmk}

The set $\Pi(\mubf)$ is actually never empty in our setting, that is, for $X,Y$ Polish spaces. If $X$ were countable it would be enough to take the independent coupling. In the general case we reason as follows.  

\begin{prop}
\label{prop:Pimu_non_empty}
Let $X,Y$ be Polish spaces, $m \in \M_+(X)$ and $\mubf \in \Mm(X \times Y)$. Then $\Pi(\mubf)$ is not empty. 
\end{prop}

\begin{proof}
We can find a Borel isomorphism of $Y$ with a subset of the real line $\R$~\cite[Theorem 1.1]{kallenberg2017random}, and thus without loss of generality we consider the case $Y = \R$.  We write $(\mubf_x)_{x \in X}$ for the disintegration of the measure $\mubf$ with respect to the first coordinate. For $t \in \R$ and $x \in X$, we define the cumulative distribution function $F(x,t) = \mubf_x((-\infty,t]) \in [0,1]$. Then for each $x \in X$ we look at $\Psi(x,\cdot) = F^{[-1]}(x,\cdot)$ its generalized inverse, or quantile function, see \cite[Definition 2.1]{SantambrogioOTAM}. The function $\Psi$ is jointly measurable as can be seen by an easy adaptation of the arguments for the Carathéodory case~\cite[Lemma 4.51]{Aliprantis2006}, using here that $\Psi(x, \cdot)$ is non-decreasing and left continuous for any $x \in X$. Calling $\lambda$ the Lebesgue measure on $[0,1]$, there holds $\Psi(x,\cdot) \# \lambda = \mubf_x$ for $m$-a.e. $x \in X$ \cite[Proposition 2.2]{SantambrogioOTAM}. For $z \in [0,1]$, we can look at the function $u_z : x \mapsto \Psi(x,z)$ which belongs to $L^0(X,\R;m)$. We define $Q$ as $Q = (z \mapsto u_z) \# \lambda$, this defines a probability distribution over $L^0(X,\R;m)$. The computations of Example~\ref{ex:parametric_Q} guarantee that $Q \in \Pi(\mubf) $ as $\mubf^Q_x = \Psi(x,\cdot) \# \lambda = \mubf_x$ at least for $m$-a.e. $x \in X$. 
\end{proof}

We now have all the tools at our disposal to state the multimarginal optimal transport problem with an infinity of marginals. 

\begin{defi}
\label{definition:MMOT_infinity}
Let $E : L^0(X,Y;m) \to [0, + \infty]$ a non-negative measurable function and $\mubf \in \Mm(X \times Y)$. We consider
\begin{equation}
\label{def:general_OT}
\mathcal{T}_E(\mubf) = \inf_{Q } \left\{ \int_{L^0(X,Y;m)} E(u) \, Q(\ddr u)  \ : \ Q \in \Prob(L^0(X,Y;m)) \text{ such that } Q \in \Pi(\mubf) \right\}.
\end{equation} 
\end{defi} 

\noindent Abstractly, it can still be seen as an infinite dimensional linear program, with the minimization of the linear function $Q \mapsto \int E \, \ddr Q$ under the convex linear constraint $Q \in \Pi(\mubf)$. 

\begin{rmk}
Requiring $E$ to be defined over the whole $L^0(X,Y;m)$ is not a limitation as it is allowed to take the value $+ \infty$: indeed a functional defined on a subset of $L^0(X,Y;m)$ can always be extended to $L^0(X,Y;m)$ by setting to $+ \infty$ outside of its natural domain of definition. Actually this will be the case for most of the examples we will consider.  
\end{rmk}

\begin{example}
\label{ex:parametric_Q_E}
We continue Example~\ref{ex:parametric_Q} in the case $X=\Omega$ is a subset of $\R^d$ endowed with $m$ the Lebesgue measure and $Y = \R^q$, while we assume that $E$ is of the form $u \mapsto \int W(\nabla u)$. Then for $Q$ given as in Example~\ref{ex:parametric_Q} we have 
\begin{equation*}
\int_{L^0(\Omega, \R^q;m)} E(u) \, Q(\ddr u) = \int_Z \left( \int_\Omega W(\nabla_x \Psi(x,z)) \, \ddr x \right) \, \lambda(\ddr z).
\end{equation*} 
Of course with this method one only ``generates'' feasible $Q$'s with no guarantee of finding the optimal one.  
\end{example}

We now want to move to our main theorem which characterizes $\T_E$ as the convex and l.s.c. envelope of $\widetilde{\T}_E$. It will require an additional assumption to get tightness of (a subset of) $\Pi(\mubf)$. The idea is that this tightness comes from an interaction between the tightness of the marginals $\mubf$ but also the coercivity of the functional $E$. When $X$ is finite the tightness of the marginals is enough, but not in the general case. We phrase the general condition as the following assumption.

\begin{asmp}
\label{asmp:reg_XYm_E}
The functional $E : L^0(X,Y;m) \to [0, + \infty]$ is l.s.c. and is such that: for every function $\psi : Y \to [0, + \infty)$ with compact sublevel sets, the function
\begin{equation*}
u \mapsto E(u) + \int_X \psi(u(x)) \, m(\ddr x)
\end{equation*}
has relatively compact sublevel sets for the $L^0(X,Y;m)$ topology.
\end{asmp}

We refer to Appendix~\ref{sec:appendix_assmp_A} for sufficient conditions for Assumption~\ref{asmp:reg_XYm_E} to hold, in particular for functionals of the form~\eqref{eq:intro_energy_generic}. It relies on the standard techniques to prove existence of a minimizer in calculus of variations. Under this assumption we now state our main result.  

\begin{theo}
\label{theo:multimarginalOT_lifting}
Let $X,Y$ be Polish spaces, $m \in \M_+(X)$ a finite non-negative measure on $X$ and a functional $E : L^0(X,Y;m) \to [0, + \infty]$ which satisfies Assumption~\ref{asmp:reg_XYm_E}. 

Then the functional $\mathcal{T}_E$ is the convex and lower semi-continuous envelope of $\widetilde{\T}_E$, and it satisfies $\mathcal{T}_E(\mubf_u) = \widetilde{\T}_E(\mubf_u) = E(u)$ for all $u \in L^0(X,Y;m)$. In addition, for any $\mubf \in \Mm(X \times Y)$ we have the duality formula
\begin{multline}
\label{eq:duality_MMOT}
\T_E(\mubf) = 
\sup_{\varphi} \Bigg\{ \int_{X \times Y} \varphi(x,y) \, \mubf(\ddr x, \ddr y) \ : \ \varphi \in C_b(X \times Y) \\
 \text{ and } \forall u \in L^0(X,Y;m), \; \int_{X} \varphi(x,u(x)) \, m (\ddr x) \leq E(u)  \Bigg\}.
\end{multline}
\end{theo}

\begin{rmk}
It is easy to check that, under the setting of Proposition~\ref{prop:envelope_linear}, the functional $\T_E$ introduced in Proposition~\ref{prop:envelope_linear} coincides with $\T_E$ introduced in~\eqref{def:general_OT}. In particular it shows that Assumption~\ref{asmp:reg_XYm_E} is not a necessary condition for $\T_E$ to coincide with the convex and l.s.c. envelope of $\widetilde{\T}_E$.
\end{rmk}

The rest of this section is dedicated to the proof of this theorem, which
will be broken down in intermediate results interesting in their own right. Let us first collect the easy results which hold even if Assumption~\ref{asmp:reg_XYm_E} does not hold.

\begin{lm}
\label{lm:basic_energy_TE}
The function $\mubf \mapsto \mathcal{T}_E(\mubf)$ is convex over $\Mm(X \times Y)$, and for every $u \in L^0(X,Y;m)$ it satisfies $\mathcal{T}_E(\mubf_u) = E(u)$.
\end{lm}

\begin{proof}
Convexity if straightforward: let $\mubf_1, \mubf_2 \in \Mm(X \times Y)$ and take also $\lambda \in [0,1]$. For any $Q_1, Q_2$ in respectively $\Pi(\mubf_1), \Pi(\mubf_2)$ (which are never empty), it is clear that $(1-\lambda) Q_1 + \lambda Q_2 \in \Pi((1-\lambda) \mubf_1 + \lambda \mubf_2)$. Thus, by linearity of the integral,
\begin{align*}
\mathcal{T}_E((1-\lambda) \mubf_1 + \lambda \mubf_2) & \leq \int_{L^0(X,Y;m)} E(u) \, [(1-\lambda) Q_1 + \lambda Q_2](\ddr u) \\
& = (1-\lambda) \int_{L^0(X,Y;m)} E(u) \, Q_1(\ddr u) + \lambda \int_{L^0(X,Y;m)} E(u) \, Q_2(\ddr u). 
\end{align*}
Taking the infimum in $Q_1, Q_2$ yields convexity. 

Then, assume that $\mubf = \mubf_u$. In this case we claim that $\Pi(\mubf_u) = \{ \delta_u \}$ and this yields $\mathcal{T}_E(\mubf_u) = E(u)$. To prove the claim, following the formulation~\eqref{eq:def_marginal_Q}, it is enough to notice that the measure $\mubf_u$ is an extreme point of the set of measures with first marginal $m$. To prove this extremality property, note that if $\mubf_u = (\mubf_1 + \mubf_2)/2$
for $\mubf_1, \mubf_2 \in \Mm(X \times Y)$, then by disintegration for $m$-a.e. $x \in X$, we have $\delta_{u(x)} =  (\mubf_{1,x}  +  \mubf_{2,x}) /2$,
being $(\mubf_{1,x})_{x \in X}$ and $(\mubf_{2,x})_{x \in X}$ the disintegrations of respectively $\mubf_1$ and $\mubf_2$ with respect to the first variable. As Dirac masses are extreme points of the set of probability distributions, $\mubf_{1,x} = \mubf_{2,x} = \delta_{u(x)}$ for $m$-a.e. $x \in X$, thus $\mubf_1 = \mubf_2 = \mubf_u$ and the conclusion follows. 
\end{proof}

Let us then state the key technical lemma which justifies the introduction of Assumption~\ref{asmp:reg_XYm_E}.

\begin{lm}
\label{lm:tightness}
Assume that Assumption~\ref{asmp:reg_XYm_E} holds. Let $\mathcal{K}$ be a tight set of $\Mm(X \times Y)$ and $C \in \R$. Then the set
\begin{equation*}
\Pi_{\leq C}(\mathcal{K}) := \left\{ Q \in \Prob(L^0(X,Y;m)) \ : \ Q \in \Pi(\mubf) \text{ for some } \mubf \in \mathcal{K} \text{ and } \int_{L^0(X,Y;m)} E(u) \, Q(\ddr u) \leq C  \right\}
\end{equation*} 
is tight for the topology of $L^0(X,Y;m)$ convergence. 
\end{lm}

\begin{proof}
Note that Assumption~\ref{asmp:reg_XYm_E} implies that the sublevel sets of $u \mapsto E(u) + \int_X \psi(u) \, \ddr m$ are not only relatively compact, but also compact, as the function $\psi$ is l.s.c. (having compact thus closed sublevel sets) and thus $u \mapsto \int_X \psi(u) \, \ddr m$ is l.s.c. for the convergence in $L^0(X,Y;m)$ by Fatou's lemma.

We then proceed to prove the lemma. 
As $\mathcal{K}$ is tight, so is its projection onto $Y$. By~\cite[Remark 5.1.5]{AGS} we can find $\psi : Y \to [0, + \infty)$ a function with compact sublevel set such that 
\begin{equation*}
\sup_{\mubf \in \mathcal{K}} \int \psi(y) \, \mubf(\ddr x, \ddr y) < + \infty.
\end{equation*}
Using the definition of $Q$ having given marginals we deduce 
\begin{equation*}
\sup_{Q \in \Pi_{\leq C}(\mathcal{K})} \left\{ \int_{L^0(X,Y;m)} \left( E(u) + \int_X \psi(u(x)) \, m(\ddr x) \right) \, Q(\ddr u) \right\} \leq C + \sup_{\mubf \in \mathcal{K}} \int \psi(y) \, \mubf(\ddr x, \ddr y) < + \infty.
\end{equation*}
Following~\cite[Remark 5.1.5]{AGS}, Assumption~\ref{asmp:reg_XYm_E} is enough to conclude to the tightness of $\Pi_{\leq C}(\mathcal{K})$. 
\end{proof}

With the help of this lemma we can prove existence of an optimal coupling $Q$ and the lower semi-continuity of $\T_E$. 

\begin{prop}
\label{prop:lsc_energy_TE}
Assume that Assumption~\ref{asmp:reg_XYm_E} holds. If $\mubf \in \Mm(X \times Y)$ is such that $\mathcal{T}_E(\mubf) < + \infty$, then the infimum in~\eqref{def:general_OT} is attained. Moreover, the function $\mathcal{T}_E$ is lower semi-continuous for the topology of narrow convergence on $\Mm(X \times Y)$.
\end{prop}

\begin{proof}
Let us first prove existence of a solution provided $\mathcal{T}_E(\mubf) < + \infty$. Take $\mubf \in \Mm(X \times Y)$ and $Q \in \Pi(\mubf)$ such that $C := \int E(u) \, Q(\ddr u) < + \infty$. Lemma~\ref{lm:tightness} implies that the set $\Pi_{\leq C}( \{ \mubf \})$ is tight thus relatively compact for the topology of narrow convergence, and it is non empty. As moreover $Q \mapsto \int E(u) \, Q(\ddr u)$ is l.s.c. for such topology thanks to Lemma~\ref{lm:lsc_implies_lsc_int}, the direct method of calculus of variations yields the result. 

Then let us move to lower semi-continuity. We take a sequence $(\mubf_n)_{n \geq 1}$ in $\Mm(X \times Y)$ which converges narrowly to $\mubf$, and we assume without loss of generality that $C := \liminf_n \mathcal{T}_E(\mubf_n) < + \infty$. For each $n$, let $Q_n \in \Pi(\mubf_n)$ attaining the infimum. Building $\mathcal{K} = \{ \mubf \} \cup \bigcup_n \{ \mubf_n \}$, this set is compact thus tight, and moreover $Q_n \in \Pi_{\leq C+1}(\mathcal{K})$, at least for $n$ along a subsequence. Lemma~\ref{lm:tightness} guarantees that we can extract a subsequence of $(Q_n)_{n \geq 1}$ which converges narrowly to $Q$, which we do not relabel. Using directly the definition of having marginals $\mubf$, it is not hard to pass to the limit the relation $Q_n \in \Pi(\mubf_n)$ to see that $Q \in \Pi(\mubf)$. Thus we get
\begin{equation*}
\mathcal{T}_E(\mubf) \leq \int_{L^0(X,Y;m)} E(u) \, Q(\ddr u) \leq \liminf_{n \to + \infty} \int_{L^0(X,Y;m)} E(u) \, Q_n(\ddr u) = \liminf_{n \to + \infty} \mathcal{T}_E(\mubf_n),
\end{equation*} 
where in the second inequality we use again the lower semi-continuity of $Q \mapsto \int E(u) \, Q(\ddr u)$ for the topology of narrow convergence, justified by Lemma~\ref{lm:lsc_implies_lsc_int}. 
\end{proof}

On the other hand the functional $\T_E$ is an upper bound for any convex and l.s.c. functional satisfying the lifting identity even if Assumption~\ref{asmp:reg_XYm_E} does not hold.

\begin{prop}
\label{prop:TE_upper_bound}
Let us take $\mathcal{T} : \Mm(X \times Y) \to [0, + \infty]$ convex and l.s.c. such that $\mathcal{T}(\mubf_u) \leq E(u)$ for all $u \in L^0(X,Y;m)$. Then $\T \leq \T_E$.
\end{prop}

\begin{proof}
Fixing $\mubf \in \Mm(X \times Y)$ we want to prove $\mathcal{T}(\mubf) \leq \mathcal{T}_E(\mubf)$. For any $Q \in \Pi(\mubf)$,
\begin{equation*}
\T(\mubf) = \T \left( \int_{L^0(X,Y;m)} \mubf_u \, Q(\ddr u) \right) \leq \int_{L^0(X,Y;m)} \T(\mubf_u) \, Q(\ddr u) \leq \int_{L^0(X,Y;m)} E(u) \, Q(\ddr u),
\end{equation*}
and taking the infimum in $Q \in \Pi(\mubf)$ yields the conclusion. The only delicate step here is the first inequality, namely Jensen's inequality, that we detail now. We obtain it following the standard proof: being convex and l.s.c. for the topology of narrow convergence, the functional $\T$ is the supremum of affine functionals, see~\eqref{eq:FenchelMoreau}: 
\begin{equation*}
\T(\mubf) = \sup_{\varphi \in C_b(X \times Y)} \left\{ \int_{X \times Y} \varphi \, \ddr \mubf - \T^*(\varphi) \right\}. 
\end{equation*}
For any $\varphi \in C_b(X \times Y)$, thanks to the definition of $\Pi(\mubf)$, 
\begin{equation*}
\int_{X \times Y} \varphi \, \ddr \mubf - \T^*(\varphi) = \int_{L^0(X,Y;m)} \left( \int_{X \times Y} \varphi \, \ddr \mubf_u - \T^*(\varphi) \right) Q(\ddr u) \leq \int_{L^0(X,Y;m)} \T(\mubf_u) \, Q(\ddr u). 
\end{equation*}
Taking the supremum in $\varphi$ in the left hand side yields Jensen's inequality and concludes the proof.
\end{proof}

\begin{proof}[\textbf{Proof Theorem~\ref{theo:multimarginalOT_lifting}}]
Lemma~\ref{lm:basic_energy_TE} and Proposition~\ref{prop:lsc_energy_TE} guarantee that $\T_E$ is convex, l.s.c. and satisfies $\T_E(\mubf_u) = E(u)$ for all $u \in L^0(X,Y;m)$. On the other hand Proposition~\ref{prop:TE_upper_bound} proves that $\T_E$ is the l.s.c. and convex envelope of $\widetilde{\T}_E$.

Eventually following the approach of~\cite{SavareSodoni2022}, we prove the second part of Theorem~\ref{theo:multimarginalOT_lifting}, that is, the duality formula~\eqref{eq:duality_MMOT}. We extend $\T_E$ by $1$-homogeneity over $\M_+(X \times Y)$, and then by $+ \infty$ for measures which are not non-negative. That is, for $\mubf \in \M(X \times Y)$ we define
\begin{equation*}
\F(\mubf) = \begin{cases}
\lambda \T_E(\mubf / \lambda) & \text{if } \mubf / \lambda \in \Mm(X \times Y) \text{ for some } \lambda \in (0, + \infty), \\
0 & \text{if } \mubf = 0, \\
+ \infty & \text{otherwise}.
\end{cases}
\end{equation*}
As $\T_E$ is convex and l.s.c., the function $\F$ is easily checked to still be convex and l.s.c. for the topology of narrow convergence. Thus it is equal to $\F^{**}$ its double Legendre transform, see~\eqref{eq:FenchelMoreau}. Computing the Legendre transform over $C_b(X \times Y)$:
\begin{align*}
\F^*(\varphi)  & = \sup_{\mubf, Q, \lambda} \left\{ \lambda \int_X \varphi \, \ddr \mubf - \lambda \int_{L^0(X,Y;m)} E \, \ddr Q \ : \ \mubf \in \Mm(X \times Y), \, Q \in \Pi(\mubf) \text{ and } \lambda \in [0, + \infty) \right\} \\
& = \sup_{Q, \lambda} \left\{ \lambda \int_{L^0(X,Y;m)} \left( \int_X \varphi(x,u(x)) \, m(\ddr x) - E(u) \right) \, Q(\ddr u) \ : \ Q \in \Prob(L^0(X,Y;m)) \text{ and } \lambda \in [0, + \infty) \right\} \\
& = \sup_{\tilde{Q}} \left\{  \int_{L^0(X,Y;m)} \left( \int_X \varphi(x,u(x)) \, m(\ddr x) - E(u) \right) \,\tilde{Q}(\ddr u) \right\}
\end{align*}
where the supremum is now taken over all $\tilde{Q}$ non-negative measure on $L^0(X,Y;m)$, and not only probability measure. Thus taking the supremum yields $+ \infty$ except if the integrand is non-positive, that is, if $\varphi \in \mathcal{C}$ where
\begin{equation*}
\mathcal{C} = \left\{ \varphi \in C_b(X \times Y) \ : \ \forall u \in L^0(X,Y;m), \; \int_{X} \varphi(x,u(x)) \, m (\ddr x) \leq E(u) \right\}.
\end{equation*}
Thus $\F^*$ is the convex indicator of $\mathcal{C}$. Writing $\T_E(\mubf) = \F^{**}(\mubf)$ yields the theorem. 
\end{proof}

\subsection{Reinterpreting already known results}
\label{sec:reinterpret}

In this section, we make briefly the link between some well established results in the literature and the framework we just presented.

\paragraph{The multimarginal optimal transport problem}
\label{sec:MMOT}

We take $X = \{ 1,2, \ldots, N \}$ a finite space. We endow it with $m$ the counting measure. In that case $L^0(X,Y;m)$ reduces to $Y^N$ with the product topology, and thus $E$ becomes a l.s.c. function from $Y^N$ to $[0, + \infty]$ that we denote by $c$. 

A measure $\mubf$ on $X \times Y$ coincides with the collection of $N$ measures $\mu_1, \ldots, \mu_N$ corresponding to $\mu_k = \mubf(\{ k \} \times Y)$ for any $k \in \{ 1, \ldots N \}$. Thus $\Mm(X \times Y)$ is in bijection with $\Prob(Y)^N$, and this is true as well for the topology on these spaces. Moreover, a probability distribution $Q$ over $Y^N$ belongs to $\Pi(\mubf)$ if its marginals are $\mu_1, \ldots, \mu_N$ following Remark~\ref{rmk:continuous_marginals}. We write $\Pi(\mu_1, \ldots, \mu_N)$ rather than $\Pi(\mubf)$ in this case. In summary, 
\begin{equation*}
\T_c(\mu_1, \ldots, \mu_N) = \min_{Q} \left\{ \int_{Y^N} c(y_1, \ldots, y_N) \, Q(\ddr y_1, \ldots, \ddr y_N) \ : \ Q \in \Prob(Y^N) \text{ and } Q \in \Pi(\mu_1, \ldots, \mu_N) \right\}
\end{equation*}
coincides with the value of the multimarginal optimal transport problem~\cite{pass2015} with cost $c$. Moreover, the cost function $c$ always satisfies Assumption~\ref{asmp:reg_XYm_E}, as actually the function $(y_1, \ldots, y_N) \mapsto \psi(y_1) + \ldots + \psi(y_N)$ has always compact sublevel sets in $Y^N$ if $\psi$ has compact sublevel sets in $Y$. This is well known for classical optimal transport: the tightness needed to get existence of an optimal transport plan comes only from tightness of the marginals, not from the coercivity of the cost function $c$. Theorem~\ref{theo:multimarginalOT_lifting} reads as follows in this context.

\begin{theo}
If $Y$ is a Polish space, then the functional $\T_c$ is the convex and lower semi-continuous envelope of the functional $\widetilde{\T}_c$ defined on $\Prob(Y)^N$ by $\widetilde{\T}(\delta_{y_1}, \ldots, \delta_{y_N}) = c(y_1, \ldots, y_N)$ for any $(y_1, \ldots, y_N) \in Y^N$ and $+\infty$ otherwise.
\end{theo}

\noindent Of course, one does not need all of our machinery to prove the result: a straightforward adaptation of the arguments of~\cite{SavareSodoni2022} would yield it, in the more general setting where $Y$ is a completely regular space.

\paragraph{Measures on curves and curves of measures}

A more elaborate case, but already well studied, is when one faces curves of measures, that is, when $X = I = [0,1]$ is a segment of $\R$ endowed with $m$ the Lebesgue measure. We also take $Y = \R^q$ a Euclidean space to simplify the presentation, more general settings are covered in the references~\cite{Lisini2007,AGS} cited in the introduction. In this case we will only consider, for $p \in (1, + \infty)$ an exponent and $k \geq 1$ an integer, the functional
\begin{equation}
E(u) = \begin{cases}
\dst{ \frac{1}{p} \int_0^1 |u^{(k)}(t)|^p \, \ddr t } & \text{if } u \in W^{k,p}(I,\R^q), \\
+ \infty & \text{otherwise},
\end{cases}
\end{equation}
being $u^{(k)}$ the $k$-th derivative of the function $u$.
As proved in Proposition~\ref{prop:checking_assumption_A} deferred in the appendix, this functional $E$ always satisfies Assumption~\ref{asmp:reg_XYm_E}. Contrary to the case of multimarginal optimal transport, here coercivity of $E$ really matters.

In this case, we are in the framework of Remark~\ref{rmk:continuous_marginals}, as any $Q$ ``measure on curves'' yielding finite cost $\int E \, \ddr Q$ is concentrated on continuous functions from $I$ to $\R^q$. We denote by $e_t : C(I, \R^q) \to \R^q$ the evaluation map at time $t \in [0,1]$. Furthermore, following Remark~\ref{rmk:continuous_marginals}, for any $Q$ with finite cost the curve of marginals $t \mapsto e_t \# Q$ will be continuous for the topology of narrow convergence. Thus the lifted functional reads as follows for a ``curve of measures'' $(\mu_t)_{t \in I}$ which is narrowly continuous
\begin{equation*}
\T_E((\mu_t)_{t \in I}) = \min_Q \left\{ \int_{W^{k,p}(I, \R^q)} \frac{1}{p} \int_0^1 |u^{(k)}(t)|^p \, \ddr t \ : \ Q \in \Prob(W^{k,p}(I, \R^q)) \text{ and } \forall t \in I, \, e_t \# Q = \mu_t \right\}. 
\end{equation*}
With Theorem~\ref{theo:multimarginalOT_lifting} we see that it is the convex and l.s.c. envelope of $\widetilde{\T}_E$.

The lifted functional $\T_E$ is central in the optimal transport theory as discussed in the introduction: the case $k=1$ relates to the velocity of curves in Wasserstein space, and the case $k=2$ and $p=2$ is linked to splines in the Wasserstein space.

\section{The Eulerian lifting for functionals depending only on first order \texorpdfstring{\\}{} derivatives}
\label{sec:lifting_TEE}

In this section we introduce the Eulerian lifting $\TEE$ and present its main properties. We will discuss the difference between $\TEE$ and $\T_E$ in the next section. We concentrate on the case where $X = \Omega$ is an open bounded subset of $\R^d$ with $d \geq 1$, while $Y = \R^q$ is a Euclidean space. We take $m$ to be the Lebesgue measure restricted to $\Omega$, and we denote integration with respect to it by $\ddr x$. We are interested in functionals $E$ of the form~\eqref{eq:intro_energy_generic}, that is, 
\begin{equation*}
u \mapsto \int_\Omega W(\nabla u(x)) \, \ddr x + \int_\Omega f(x,u(x)) \, \ddr x.
\end{equation*} 
The second part $u \mapsto \int_\Omega f(x,u(x)) \, \ddr x$ of the functional has already been dealt with in Proposition~\ref{prop:envelope_linear}, and can be lifted without any problem as it becomes linear. The more involved part is $u \mapsto \int W(\nabla u)$, so we will consider only this case in this section. 

Some lifted version $\TEE$ mimicking the ``fluid dynamic'' formulation~\eqref{eq:intro_dynamic_curves} for curves of measures, or ``Eulerian'' formulation, has already been proposed as discussed in the introduction. We will introduce $\TEE$, present some results about well-posedness, dual formulation and lifting identity. Most of the results can already be found in the literature, either in a slightly more specific or slightly more general setting. We only require the following on $W$.

\begin{asmp}
\label{asmp:W}
The function $W : \R^{qd} \to [0, + \infty)$ is convex and finite everywhere on $\R^{qd}$. Moreover, it grows at least like $|v|^p$ for $p \in [1, + \infty)$ in the sense
\begin{equation*}
W(v)  \geq  C_1 |v|^p - C_2,
\end{equation*} 
for some $C_1 > 0$ and $C_2 \geq 0$. Here, $| \cdot |$ denotes any norm on the space of matrices. 
\end{asmp}

\begin{rmk}
It could be possible to allow $W$ to depend on $x$, that is, to take a functional of the form $u \mapsto \int_\Omega W(x, \nabla (u)) \, \ddr x$. Though we think most of the results of this section to carry through without any major difficulty, it is not the case for the next section as Lemma~\ref{lm:regularization_measure_valued} would break down. Thus we stick to  $W$ independent of $x$. 
\end{rmk}

For such a function $W$, we define its recession function~\cite[Definition 2.32]{AFP}, written $W'_\infty$, by 
\begin{equation*}
W'_\infty(v) = \lim_{t \to + \infty} \frac{W(tv)}{t},
\end{equation*}
which is well-defined thanks the convexity of $W$. 
It is $1$-homoegenous, convex, and moreover under Assumption~\ref{asmp:W}
\begin{equation}
\label{eq:lower_bound_recession}
W'_\infty(v) \geq \begin{cases}
C_1 |v| & \text{if } p =1, \\
+ \infty & \text{if } p >1, \, v \neq 0.
\end{cases}
\end{equation} 
We also define its Legendre transform $W^*(r) = \sup_v r \cdot v - W(v)$. Under Assumption~\ref{asmp:W} we can check that $W^*$ is defined at least on a neighborhood of $0$.

\begin{defi}
\label{defi:E_from_W}
Let $W$ which satisfies Assumption~\ref{asmp:W}. We define $E(u) \in [0,+ \infty]$ for $u \in L^0(\Omega,\R^q;m)$ as 
\begin{equation*}
E(u) = \begin{cases}
\dst{\int_\Omega W(\nabla^a u(x)) \, \ddr x + \int_\Omega W'_\infty \left( \frac{\ddr \nabla^s u}{\ddr |\nabla^s u|} \right) \ddr |\nabla^s u|}  & \text{if } u \in \mathrm{BV}(\Omega,\R^q), \\
+ \infty & \text{otherwise},
\end{cases}
\end{equation*}
being $\nabla u = (\nabla^a u)m + \nabla^s u$ the Radon-Nikodym decomposition of the distributional Jacobian matrix of $u$ with respect to the Lebesgue measure. 
\end{defi}

\noindent In the case $p>1$, the second term forces $\nabla^s u = 0$ for $E(u)$ to be finite, and thus $E$ has the simpler expression
\begin{equation*}
E(u) = \begin{cases}
\dst{\int_\Omega W(\nabla u(x)) \, \ddr x } & \text{if } u \in W^{1,p}(\Omega,\R^q), \\
+ \infty & \text{otherwise}.
\end{cases}
\end{equation*}
The case $p=1$ requires special care as usual for functions with sublinear growth in the calculus of variations. As we prove in Appendix~\ref{sec:appendix_assmp_A}, specifically in Proposition~\ref{prop:checking_assumption_A} and Proposition~\ref{prop:checking_assumption_A_p_=1}, the functional $E$ is l.s.c. for the topology of $L^0(\Omega,\R^q;m)$ convergence and satisfies Assumption~\ref{asmp:reg_XYm_E}, including for the case $p=1$. 

Next we move on to the definition of $\TEE$, built by mimicking the fluid dynamic formulation for measure-valued curves~\eqref{eq:intro_dynamic_curves}. We add an additional variable $J$, a $q \times d$ matrix-valued measure on $\Omega \times \R^q$, that we see as the collection of $q \times d$ signed measures. If it exists, its density with respect to $\mubf$, called $v$ in the introduction, could be interpreted as the ``density of Jacobian matrix''. 

\begin{defi}
\label{def:energyBW}
For $\mubf \in \M_+(\Omega \times \R^q)$ and $J \in \M(\Omega \times \R^q)^{qd}$, writing $J = v \mubf + J^\perp$ the Radon-Nikodym decomposition of $J$ with respect to $\mubf$, we define
\begin{equation*}
\B_W(\mubf, J) = \int_{\Omega \times \R^q} W(v(x,y)) \, \mubf(\ddr x, \ddr y) + \int_{\Omega \times \R^q}  W'_\infty \left( \frac{\ddr J^\perp}{\ddr |J^\perp|} \right) \, \ddr |J^\perp|.
\end{equation*}
\end{defi}

\noindent Note that if $p>1$ then thanks to~\eqref{eq:lower_bound_recession} the second term is the convex indicator of the set $|J^\perp|(\Omega \times \R^q) = 0$, that is, of the constraint $J \ll \mubf$. We defer to Appendix~\ref{sec:appendix_BW} for some standard properties of the functional $\B_W$. 

Eventually we add a constraint $\nabla_x \mubf + \nabla_y \cdot J = 0$ and minimize over $J$ to obtain $\TEE$. Expanded in coordinates it reads:
\begin{equation*}
\forall j \in \{1, 2, \ldots, d \}, \quad \partial_{x_j} \mubf + \sum_{i=1}^q \partial_{y_i} J_{ij} = 0,
\end{equation*}
where $J_{ij}$ for $i \in \{1, 2, \ldots, q \}$ and $j \in \{1, 2, \ldots, d \}$ correspond to the entries of the matrix-valued measure $J$. We will impose this constraint in a weak sense: to express it we will need test functions $\varphi = (\varphi_j)_{1 \leq j \leq d} : \Omega \times \R^q \to \R^d$, for which we define the differential operators $\nabla_x \cdot \varphi$ and $\nabla_y \varphi$, respectively real-valued and $q \times d$ matrix-valued, by: 
\begin{equation*}
\nabla_x \cdot \varphi = \sum_{j=1}^d \partial_{x_j} \varphi_j, \qquad \nabla_y \varphi = \left( \partial_{y_i} \varphi_j \right)_{1 \leq i \leq q, \, 1 \leq j \leq d}.
\end{equation*}

\begin{rmk}
\label{rmk:transpose_varphi}
We do a small abuse of notations in order to lighten them: here we see $\nabla_y \varphi$ as a $q \times d$ matrix, so it would rather correspond to the \emph{transpose} of the Jacobian matrix of the map $\varphi$ with respect to the $y$ variables, which is rather a $q \times d$ matrix. Thanks to this convention, $\nabla_y \varphi$ and $J$ have both the same dimension.
\end{rmk}

\begin{defi}
\label{defi:TEE}
For $\mubf  \in \Mm(\Omega \times \R^q)$ we define
\begin{equation*}
\TEE(\mubf) = \min_J \left\{ \B_W(\mubf, J) \ : \ J \in \M(\Omega \times \R^q)^{qd} \text{ such that } \nabla_x \mubf + \nabla_y \cdot J = 0  \right\},
\end{equation*} 
where the equation $\nabla_x \mubf + \nabla_y \cdot J = 0$ is understood in a weak sense, that is, for every smooth compactly supported test function $\varphi : \Omega \times \R^q \to \R^d$,
\begin{equation*}
\int_{\Omega \times \R^q} \nabla_x \cdot \varphi(x,y) \, \mubf(\ddr x, \ddr y) + \int_{\Omega \times \R^q} \nabla_y \varphi(x,y) \cdot J(\ddr x, \ddr y) = 0.
\end{equation*}  
\end{defi}

\noindent Following Remark~\ref{rmk:transpose_varphi}, this equation reads in coordinates:
\begin{equation*}
\sum_{j=1}^d \int_{\Omega \times \R^q} \partial_{x_j} \varphi_j \, \ddr \mubf + \sum_{i=1}^q \sum_{j=1}^d \int_{\Omega \times \R^q} \partial_{y_i} \varphi_j \, \ddr J_{ij} = 0.
\end{equation*}

Now that we have our definition, there are a few properties that we want to collect: first that $\TEE$ is convex and l.s.c., and then a duality formula which enable to prove that $\TEE$ coincides with $\widetilde{\T}_E$ on maps $\mubf_u$ with $u \in L^0(X,Y;m)$.

\begin{lm}
\label{lemma:existence_tangent_J}
For $\mubf  \in \Mm(\Omega \times \R^q)$ such that $\TEE(\mubf, \Omega) < + \infty$, there exists $J \in \M(\Omega \times \R^q)^{qd}$ such that $\nabla_x \mubf + \nabla_y \cdot J = 0$ in a weak sense and $\TEE(\mubf) = \B_W(\mubf,J)$. Such a measure $J$ is said \emph{tangent} to $\mubf$.
\end{lm}

\noindent If $W$ is strictly convex, then there exists a unique tangent $J$.

\begin{proof}
We use the direct method of calculus of variations. For a minimizing sequence $(J_n)_{n \in \N}$ the estimate~\eqref{eq:coercivity_BW} guarantees that the mass of $|J_n|$ is bounded, thus weak-$\star$ converges. Lower semi-continuity of $\B_W$ together with closedness of the continuity equation in its weak form with respect to weak-$\star$ convergence enable to conclude.
\end{proof}

\begin{prop}
\label{prop:TEE_convex_lsc}
The functional $\TEE$ is convex and lower semi-continuous over $\Mm(\Omega \times \R^q)$. 
\end{prop}

\begin{proof}
Convexity is straightforward. Take $\mubf_1, \mubf_2 \in \Mm(\Omega \times \R^q)$ and take $J_1, J_2 \in \M(\Omega \times \R^q)^{qd}$ tangent to respectively $\mubf_1$ and $\mubf_2$. If $\mubf = (1- \lambda) \mubf_1 + \lambda \mubf_2$ for $\lambda \in [0,1]$, then with $J = (1- \lambda) J_1 + \lambda J_2$ it is clear by linearity that $\nabla_x \mubf + \nabla_y \cdot J = 0$. Moreover, by joint convexity of $\B_W$ (see Proposition~\ref{prop:dual_BW}) we have $\B_W(\mubf,J) \leq (1- \lambda) \B_W(\mubf_1,J_1) + \lambda \B_W(\mubf_2,J_2)$ and this is enough to conclude. 

Then we proceed to lower semi-continuity. To that end, if $(\mubf_n)_{n \geq 1}$ is a sequence in $\Mm(\Omega \times \R^q)$ converging narrowly to $\mubf$ with $\liminf_n \TEE(\mubf_n) < + \infty$, we take $J_n$ the tangent measure to $\mubf_n$ for any $n$. Then, again thanks to~\eqref{eq:coercivity_BW} the sequence of measure $(J_n)_{n \in \N}$ converges weak-$\star$, up to extraction, to some $J$. Taking the weak limit of $\nabla_x \mubf_n + \nabla_y \cdot J_n = 0$ yields $\nabla_x \mubf + \nabla_y \cdot J = 0$, and the (joint) lower semi-continuity of $\B_W$ (see Proposition~\ref{prop:dual_BW}) yields
\begin{equation*}
\TEE(\mubf) \leq \B_W(\mubf,J) \leq \liminf_{n \to + \infty} \B_W(\mubf_n,J_n)  = \liminf_{n \to + \infty} \TEE(\mubf_n). \qedhere
\end{equation*} 
\end{proof}

We move to a duality formula, which is not surprising as the problem defining $\TEE(\mubf)$ is a convex optimization problem with a linear constraint.  The duality can be thought as ``$\mubf$ is given and only $J$ is the optimization variable'', and is an easy adaptation of the techniques we learned in~\cite{santambrogio2018regularity}, originally taken from~\cite{bouchitte2001characterization}.

\begin{theo}[duality for $\TEE$]
\label{theo:duality_TWd}
If $W$ satisfies Assumption~\ref{asmp:W}, for any $\mubf \in \Mm(\Omega \times \R^q)$, there holds
\begin{multline}
\label{eq:duality_TWd}
\TEE(\mubf) = \sup_{\varphi} \Bigg\{ - \int_{\Omega \times \R^q} \left( \nabla_x \cdot \varphi + W^*(\nabla_y \varphi)  \right) \, \ddr \mubf \ : \ \varphi \in C^1_c(\Omega \times \R^q, \R^{d}) \\
 \text{ and } W^*(\nabla_y \varphi) < + \infty \text{ on } \Omega \times \R^q   \Bigg\}.
\end{multline}
\end{theo}

\begin{proof}
Let us introduce $V$ the set of $\R^d$-valued distributions over $\Omega \times \R^q$ which can be written $\nabla_y \cdot J$ for a measure $J \in \M(\Omega \times \R^q)^{qd}$. We also introduce $V' = C^1_c(\Omega \times \R^q, \R^d)$ the set of $C^1$ and compactly supported functions in $\Omega \times \R^q$ and valued in $\R^d$. There is a natural non-degenerate pairing $\langle \cdot, \cdot \rangle_{V,V'}$ between $V$ and $V'$: the one between distributions and functions.

We introduce over $V$ the function $\F : V \to [0,+\infty]$ defined by 
\begin{equation*}
\F(h) = \min_J \left\{ \B_W(\mubf, J) \ : \ J \in \M(\Omega \times \R^q)^{qd} \text{ such that } \nabla_x \mubf + \nabla_y \cdot J = h  \right\},
\end{equation*}
in such a way that $\F(0) = \TEE(\mubf)$. The same arguments as in Proposition~\ref{prop:TEE_convex_lsc} prove that $\F$ is convex and l.s.c. for the $\sigma(V,V')$ topology, hence it is equal to its double Legendre transform. The Legendre transform reads, for $\varphi \in V'$,
\begin{equation*}
\F^*(\varphi) = \sup_{h \in V} \;  \langle h, \varphi \rangle_{V,V'} - \F(h) = \sup_{\tilde{J}, J \in \M(\Omega \times \R^q)^{qd}} \left\{ \langle \nabla_y \cdot \tilde{J}, \varphi \rangle_{V,V'} - \B_W(\mubf, J) \ : \  \nabla_x \mubf + \nabla_y \cdot J = \nabla_y \cdot \tilde{J} \right\},
\end{equation*}
where we just expanded the definition of $\F$ and wrote $h = \nabla_y \cdot \tilde{J}$. We can eliminate $\tilde{J}$ using the constraint on $\nabla_x \mubf + \nabla_y \cdot J = \nabla_y \cdot \tilde{J}$ and then proceed to an integration by parts:
\begin{align*}
\F^*(\varphi) &= \sup_{J \in \M(\Omega \times \R^q)^{qd}} \left\{  \int_{\Omega \times \R^q} \nabla_x \cdot \varphi \,  \ddr \mubf + \int_{\Omega \times \R^q} \nabla_y \varphi \cdot \ddr J - \B_W(\mubf, J)  \right\} \\
&=  \int_{\Omega \times \R^q} \nabla_x \cdot \varphi \,  \ddr \mubf + \sup_{J \in \M(\Omega \times \R^q)^{qd}} \left\{  \int_{\Omega \times \R^q} \nabla_y \varphi \cdot \ddr J - \B_W(\mubf, J)  \right\} \\
& = \begin{cases}  \dst{\int_{\Omega \times \R^q} \nabla_x \cdot \varphi \,  \ddr \mubf + \int_{\Omega \times \R^q} W^*(\nabla_y \varphi) \, \ddr \mubf} & \text{if } W^*(\nabla_y \varphi) < + \infty \text{ on } \Omega \times \R^q, \\
+ \infty & \text{otherwise},
\end{cases}
\end{align*}
where the last equality comes from the expression of the Legendre transform of $\B_W$ with respect to its second variable, which we obtain from~\eqref{eq:dual_BW_only_J}, see Appendix~\ref{sec:appendix_BW}. 

With the Fenchel-Moreau identity~\eqref{eq:FenchelMoreau} we get $\TEE(\mubf) = \F(0) = \F^{**}(0) = \sup_{V'} - \F^*$ which is our claim. 
\end{proof}


\begin{crl}
\label{crl:lifting_T_dyn}
Under Assumption~\ref{asmp:W}, and with $E$ defined as in Definition~\ref{defi:E_from_W}, for any $u \in L^0(\Omega,\R^q;m)$ there holds $\TEE(\mubf_u) = \widetilde{\T}_E(\mubf_u) = E(u)$.
\end{crl}

\begin{proof}
First assume that $u \in L^0(\Omega,\R^q;m)$ is such that $E(u) < + \infty$. We use the ``primal'' formulation, that is, Definition~\ref{defi:TEE}. Calling $\nabla u$ the distributional Jacobian matrix of $u$ (which is at least a measure), we define $J_u \in \M(\Omega \times \R^q)^{qd}$ via its action on test functions $\varphi \in C_b(\Omega \times \R^q, \R^{qd})$: 
\begin{equation*}
\int_{\Omega \times \R^q} \varphi \cdot \ddr J = \int_\Omega \varphi(x,u(x)) \cdot (\nabla u) (\ddr x).
\end{equation*}
With this definition, it is now standard to check that $\nabla_x \mubf_u + \nabla_y \cdot J_u = 0$ in the weak sense: we leave it as an exercise to the reader, and we refer to~\cite[Proposition 5.2]{Lavenant2019} where we did it. Moreover, writing $\nabla u = (\nabla^a u) m + \nabla^s u$ the Radon-Nikodym decomposition of the (distributional) Jacobian matrix of $u$ with respect to the Lebesgue measure $m$, we obtain the Radon-Nikodym decomposition of $J_u$ with respect to $\mubf_u$: with the notations of Definition~\ref{def:energyBW}, we have $v(x,u(x)) = \nabla^a u(x)$ while $J^\perp$ is the pushforward of $\nabla^s u$ by the map $x \mapsto (x,u(x))$. Thus following the definitions we get $\B_W(\mubf_u, J_u) = E(u)$. We conclude to our first inequality $\TEE(\mubf_u) \leq \B_W(\mubf_u, J_u) = E(u)$.

We then proceed to prove the converse inequality, and for that we rely on Theorem~\ref{theo:duality_TWd}. We take $u \in L^0(\Omega, \R^q;m)$ with $\TEE(\mubf_u) < + \infty$. We restrict to test functions $\varphi(x,y) = \psi(x)^\top y$ where $y \in \R^q$ and $\psi : \Omega \to \R^{qd}$ is a $C^1$ and compactly supported $q \times d$ matrix-valued field, so that $\psi(x)^\top y$ is a matrix vector product between the transpose of $\psi$ and $y \in \R^q$. Theorem~\ref{theo:duality_TWd} yields that for any such $\psi$ 
\begin{equation*}
- \int_\Omega (\nabla_x \cdot \psi(x) ) \cdot u(x) \, \ddr x - \int_\Omega W^*(\psi(x)) \, \ddr x \leq  \TEE(\mubf_u).  
\end{equation*}  
As $W^*$ finite in a neighborhood of $0$ we see that $- \int_\Omega (\nabla_x \cdot \psi) \cdot u$ is finite if $\sup |\psi|$ is small enough, thus we deduce that the weak derivative of $u$ in the sense of distributions is at least a measure. Proceeding with an integration by parts, we obtain that for any $\psi$ as above,
\begin{equation*}
\int_\Omega \psi(x) \cdot (\nabla u)(\ddr x)  - \int_\Omega W^*(\psi(x)) \, \ddr x \leq  \TEE(\mubf_u). 
\end{equation*}
Taking the supremum in $\psi$ in the left hand side yields $E(u)$, this is classical result of convex analysis that can be found for instance in the proof~\cite[Proposition 7.7]{SantambrogioOTAM}. Thus we obtain the other inequality $E(u) \leq \TEE(\mubf_u)$.
\end{proof}

\begin{rmk}
\label{rmk:lifting_TV_dual}
An important case is when $W$ is an operator norm on the space of matrices, and it yields a TV penalization. Specifically, let $N_{\R^d}$ and $N_{\R^q}$ be two norms on respectively $\R^d$ and $\R^q$. For a norm $N$ on a Euclidean space we denote by $N^*$ its dual norm, defined as $N^*(a) = \sup_{b} \{ a \cdot b  \ :  \ N(b) \leq 1 \}$, being $a \cdot b$ the canonical scalar product between vectors. The dual of the dual norm is the norm itself. 
We denote by $\| \cdot \|_{N_{\R^d} \to N_{\R^q}}$ the operator norm on $q \times d$ matrices:
\begin{equation*}
\| a \|_{N_{\R^d} \to N_{\R^q}} = \sup_{b \in \R^d} \{ N_{\R^q}(ab) \ : \ N_{\R^d}(b) \leq 1 \}.
\end{equation*}
One can check that $\| a^\top \|_{N^*_{\R^q} \to N^*_{\R^d}} = \| a \|_{N_{\R^d} \to N_{\R^q}}$ for any $q \times d$ matrix.
Taking $W = \| \cdot \|_{N_{\R^d} \to N_{\R^q}}$, which satisfies Assumption~\ref{asmp:W} with $p = 1$, a direct computation yields
\begin{equation*}
W^*(r) = \begin{cases}
0 & \text{if } \| r \|_{N^*_{\R^d} \to N^*_{\R^q}} \leq 1, \\
+ \infty & \text{otherwise}. 
\end{cases}
\end{equation*}
Thus in this context Theorem~\ref{theo:duality_TWd} implies
\begin{equation}
\label{eq:dual_TEE_norm}
\T_{E, \mathrm{Eul}} (\mubf) = \sup_{\varphi} \left\{ - \int_{\Omega \times \R^q} \nabla_x \cdot \varphi \, \ddr \mubf \ : \ \varphi \in C^1_c(\Omega \times \R^q, \R^d) \text{ and } \forall x \in \Omega, \, \varphi(x,\cdot) \text{ is } 1 \text{-Lipschitz} \right\},
\end{equation}
where $1$-Lipschitz is understood when putting $N_{\R^q}$ on the domain $\R^q$ and $N_{\R^d}$ on the codomain $\R^d$. Indeed thanks to the computations above, and given that $\nabla_y \varphi$ is understood as the \emph{transpose} of the Jacobian matrix of $\varphi$ (see Remark~\ref{rmk:transpose_varphi}), it is equivalent to requiring $W^*(\nabla_y \varphi) < + \infty$ everywhere on $\Omega \times \R^q$. 
The formulation~\eqref{eq:dual_TEE_norm} is how the Eulerian lifting was actually defined in the works~\cite{laude2016sublabel, Vogt2018, vogt2019lifting}.
\end{rmk}

\begin{rmk}
In the case $W$ is the squared Froebenius norm and the measure $\mubf$ is supported in $\Omega \times D$ for $D \subseteq \R^q$ convex and bounded the functional $\TEE$ has a metric counterpart: it corresponds to the Dirichlet energy in the sense of Korevaar, Schoen and Jost~\cite{Korevaar1993,Jost1994} of the map $x \mapsto (\mubf_x)_{x \in X}$ seen as a map valued in the metric space $\Prob(D)$ endowed with the quadratic Wasserstein distance~\cite[Section 3.4]{Lavenant2019}. This can be read as the generalization of formula~\eqref{eq:intro_metric_curve} which holds for measure-valued curves. Moreover, in this case it is possible to make sense of the trace of the measure-valued map on the boundary of $\Omega$, provided the latter is e.g. Lipschitz. 
\end{rmk}

\section{On the difference between the Lagrangian and the Eulerian lifting}
\label{sec:difference_liftings}

In this section we still take $X = \Omega$ an open bounded subset of $\R^d$ with $d \geq 1$, while $Y = \R^q$ is a Euclidean space. We take $m$ to be the Lebesgue measure restricted to $\Omega$, and we denote integration with respect to it by $\ddr x$. As in the previous section we focus on functionals $E : u \mapsto \int_\Omega W(\nabla u)$ for $W$ convex.

We now have two liftings, $\T_E$ and $\TEE$, respectively the Lagrangian and the Eulerian one, and we know that the former is the convex and l.s.c. envelope of $\widetilde{\T}_E$, see Theorem~\ref{theo:multimarginalOT_lifting}.
In particular this theorem implies $\TEE \leq \T_E$ thanks to Proposition~\ref{prop:TEE_convex_lsc} and Corollary~\ref{crl:lifting_T_dyn}.
In this section, we want to explain that the inequality is strict, and provide a characterization of $\TEE$. A key characteristic of functionals $E$ of the type~\eqref{eq:intro_energy_generic} is that they are \emph{local} and can be seen as \emph{measures}, that is, they decompose additively when restricted to disjoint open sets. Generically, this property is not preserved by the lifting $\T_E$, while on the other hand $\TEE$ is the largest lifting also preserving this property. 

We will first discuss how to formalize the dependence of functionals on sets, borrowing for that purpose concepts, definitions and results coming from the literature on $\Gamma$-convergence~\cite[Chapters 14 and 15]{dalMaso2012introduction}. Next, we provide examples (already present in our previous work~\cite{Lavenant2019}) justifying that $\T_E$ is not additive. Eventually, we characterize $\TEE$ as the convex, l.s.c. and subadditive envelope.

\subsection{Localization of functionals}
\label{sec:localization}

We fix $\Omega$ an open subset of $\R^d$. We denote by $\A(\Omega)$ the set of open subset of $\Omega$. We first recall some definitions, in case we have a functional $F : Z \times \A(\Omega) \to [0, + \infty]$, which depends both on a function space $Z$ (generically a Polish space) and on a open set in $\A(\Omega)$.

\begin{defi}
\label{defi:localized_functional}
A functional $F : Z \times \A(\Omega) \to [0, + \infty]$ is said:
\begin{enumerate}[label=(\roman*)]
\item \emph{lower semi-continuous} if $Z$ is a topological space and for every $A \in \A(\Omega)$, the function $z \in Z \mapsto F(z,A)$ is lower semi-continuous,
\item \emph{convex} if $Z$ is a convex subset of a vector space and if for every $A \in \A(\Omega)$, the function $z \in Z \mapsto F(z,A)$ is convex,
\item \emph{local} if $Z$ is a subset of $L^0(\Omega,\R^q;m)$, and $u=v$ a.e. on $A \in \A(\Omega)$ implies $F(u,A) = F(v,A)$,
\item \emph{increasing} if $F(z, A_1) \leq F(z,A_2)$ for $z \in Z$ and $A_1 \subseteq A_2$ open sets of $\Omega$,
\item \emph{inner regular} if for every $z \in Z$ and every $A \in \A(\Omega)$ there holds
\begin{equation*}
F(z,A) = \sup_{B} \left\{ F(z,B) \ : \ B \in \A(\Omega), \; B \text{ compactly included in } A \right\},
\end{equation*}
\item \emph{subadditive} if $F(z, A_1 \cup A_2) \leq F(z,A_1)+F(z, A_2)$ for $z \in Z$ and $A_1, A_2$ open sets of $\Omega$,
\item \emph{superadditive} if $F(z, A_1 \cup A_2) \geq F(z,A_1)+F(z, A_2)$ for $z \in Z$ and $A_1, A_2$ disjoint open sets of $\Omega$,  
\item a \emph{measure} if for every $z \in Z$ the set function $A \mapsto F(z,A)$ is the restriction to $\A(\Omega)$ of a regular Borel measure on the Borel $\sigma$-algebra of $\Omega$.
\end{enumerate}
\end{defi}

\noindent The first two properties are only concerned with the (classical) functional $F(\cdot,A)$ for a fixed $A$. Starting from the fourth one, all properties are only concerned with the set function $F(z,\cdot)$ for fixed $z$. Only the third one, locality, displays an interplay between the two arguments of the function $F$.

Note that a measure is always subadditive, superadditive, increasing and inner regular. A celebrated result \cite[Theorem 14.23]{dalMaso2012introduction} guarantees that an increasing functional which is subadditive, superadditive and inner regular is in fact a measure.

\begin{example}
\label{ex:localization_E}
Let $W : \R^{qd} \to [0, + \infty)$ a convex function satisfying Assumption~\ref{asmp:W}. We define for $u \in L^0(\Omega,\R^q;m)$ and $A \in \A(\Omega)$, 
\begin{equation*}
E(u,A) = \begin{cases}
\dst{\int_{A}} W(\nabla u(x)) \, \ddr x & \text{if } u \in W^{1,p}(A, \R^q), \\
+ \infty & \text{otherwise}
\end{cases}
\end{equation*}
in the case $p > 1$, with the natural extension (following Definition~\ref{defi:E_from_W}) in the case $p=1$. Then this function is convex, l.s.c. and a measure, see~\cite[Example 15.4]{dalMaso2012introduction}.
\end{example}  

We then want to lift such localized functionals. We apply the same localization procedure and we see the lifted functional as defined over the product space $\Mm(\Omega \times \R^q) \times \A(\Omega)$. All points of Definition~\ref{defi:localized_functional} carry through, expect the one on locality. We can adapt easily by saying that $\mathcal{T} : \Mm(\Omega \times \R^q) \times \A(\Omega) \to [0,+\infty]$ is local if for every $\mubf, \nubf$ measures on $\Mm(\Omega \times \R^q)$ and for every open set $A \in \A(\Omega)$, if $\mubf = \nubf$ when restricted to $A \times \R^q$ then $\mathcal{T}(\mubf,A) = \mathcal{T}(\nubf,A)$.

\paragraph{Localized lifting $\T_E$}

There is a clear way to adapt Definition~\ref{definition:MMOT_infinity} to lift  a localized functional $E$ defined on $L^0(\Omega,\R^q;m) \times \A(\Omega)$ into a functional $\T_E$ on $\Mm(\Omega \times \R^q) \times \A(\Omega)$: for $A \in \A(\Omega)$ and $\mubf \in \Mm(\Omega \times \R^d)$ we define 
\begin{equation*}
\T_E( \mubf, A) = \T_{E(\cdot, A)}(\mubf) = \inf_{Q } \left\{ \int_{L^0(\Omega,\R^q;m)} E(u,A) \, Q(\ddr u)  \ : \ Q \in \Prob(L^0(\Omega,\R^q;m)) \text{ such that } Q \in \Pi(\mubf) \right\},
\end{equation*}
that is, we lift $E(\cdot, A)$ into $\T_E(\cdot, A)$, for any fixed $A \in \A(\Omega)$. If $E$ is local we can derive an alternative expression. 

\begin{lm}
\label{lm:alternate_representation_lifting_TE}
Assume $E : L^0(\Omega, \R^q;m) \times \A(\Omega) \to [0,+\infty]$ is local. Then, for any set $A \in \A(\Omega)$ and any $\mubf \in \Mm(\Omega \times \R^q)$, calling $\mubf |_A$ the restriction of $\mubf$ to $A \times \R^q$,
\begin{equation}
\label{eq:alternate_representation_lifting_TE}
\T_E( \mubf, A) = \inf_{\tilde{Q}} \left\{ \int_{L^0(A,\R^q;m)} E(u,A) \, \tilde{Q}(\ddr u)  \ : \ \tilde{Q} \in \Prob(L^0(A,\R^q;m)) \text{ such that } \tilde{Q} \in \Pi(\mubf |_A) \right\}.
\end{equation}
\end{lm}

\begin{proof}
The left hand side is always larger then the right hand side as, for any $Q \in \Pi(\mubf)$, we can look at $\tilde{Q} = e_A \# Q$, being $e_A : L^0(\Omega,\R^q;m) \to L^0(A,\R^q;m)$ the restriction operator to $A$. It gives $\tilde{Q} \in \Pi(\mubf |_A)$, and $\int E(\cdot,A) \, \ddr Q = \int E(\cdot,A) \, \ddr \tilde{Q}$ by locality of $E$. 

On the other hand, consider $\tilde{Q} \in \Prob(L^0(A,\R^q;m))$ such that $\tilde{Q} \in \Pi(\mubf |_A)$. We write $A^c = \Omega \setminus A$, and $\mubf|_{A^c}$ for the restriction of $\mubf$ to $A^c \times \R^q$. Taking any $\hat{Q} \in \Prob(L^0(A^c, \R^q;m))$ such that $\hat{Q} \in \Pi(\mubf|_{A^c})$ (the latter is not empty, see Proposition~\ref{prop:Pimu_non_empty}), we can build $Q = \tilde{Q} \otimes \hat{Q} \in \Prob(L^0(\Omega,\R^q;m))$: that is, a $u$ drawn according to $Q$ is such that the restriction of $u$ to $A$ (resp. $A^c$) follows $\tilde{Q}$ (resp. $\hat{Q}$), and the values on $A$ and $A^c$ are independent. The probability measure $Q$ satisfies $Q \in \Pi(\mubf)$ and $\int E(\cdot,A) \, \ddr Q = \int E(\cdot,A) \, \ddr \tilde{Q}$, again by locality of $E$. That gives us that the left hand side is always smaller then the right hand side in~\eqref{eq:alternate_representation_lifting_TE}. 
\end{proof}

We collect the properties that the lifting $\T_E$ inherits from the functional $E$ in the following proposition. 

\begin{prop}
\label{prop:lifting_TE_properties}
Let $E : L^0(\Omega,\R^q;m) \times \A(\Omega) \to [0, + \infty]$ a localized functional. The localized lifting $\T_E : \Mm(\Omega \times \R^q) \times \A(\Omega) \to [0,+\infty]$ satisfies the following properties.
\begin{enumerate}[label=(\roman*)]
\item It is convex and satisfies $\T_E(\mubf_u,A) = E(u,A)$ for any $u \in L^0(\Omega,\R^q;m)$ and any $A \in \A(\Omega)$.
\item If functional $E$ is local, then so is $\mathcal{T}_E$.
\item If functional $E$ is local and for any $A \in \A(\Omega)$, the functional $E(\cdot, A)$ when restricted to $L^0(A,\R^q;m)$ satisfies Assumption~\ref{asmp:reg_XYm_E}, then $\T_E$ is l.s.c.
\item If the functional $E$ is superadditive, then so is $\mathcal{T}_E$.
\end{enumerate} 
\end{prop}

\noindent Note that, for the functionals $E$ of Example~\ref{ex:localization_E}, we obtain that $\T_E$ is convex, l.s.c., local, superadditive and satisfies the lifting identity.

\begin{proof}
Point (i) is the same as Lemma~\ref{lm:basic_energy_TE}.

Point (ii) is a direct consequence of~\eqref{eq:alternate_representation_lifting_TE}.

Point (iii) follows from~\eqref{eq:alternate_representation_lifting_TE} together with Proposition~\ref{prop:lsc_energy_TE}. 

Eventually for (iv), if $Q \in \Pi(\mubf)$ and $A_1$, $A_2$ are disjoint open sets in $\Omega$, 
\begin{align*}
\int_{L^0(\Omega,\R^q;m)} E(u,A_1 \cup A_2) \, Q(\ddr u) & \geq \int_{L^0(\Omega,\R^q;m)} E(u,A_1 ) \, Q(\ddr u) + \int_{L^0(\Omega,\R^q;m)} E(u,A_2) \, Q(\ddr u) \\
& \geq \T_E(\mubf,A_1) + \T_E(\mubf, A_2),
\end{align*}
where we first used superadditivity of $E$ and then the definition of $\T_E$. Taking the infimum over $Q$ yields the conclusion.
\end{proof}

As we will see in the sequel, even if the functional $E$ is a measure, it is \emph{not} guaranteed for $\T_E$ to be a measure and the best we can conclude is superadditivity. 

\paragraph{Localized lifting $\TEE$}

When we take $E$ as in Example~\ref{ex:localization_E} the localized version of $\TEE$ is very easy build too. Following Definition~\ref{defi:TEE}, we set for $\mubf \in \Mm(\Omega \times \R^q)$ and $A \in \A(\Omega)$,
\begin{equation}
\label{eq:TEE_localized}
\TEE(\mubf,A) = \min_J \left\{ \B_W(\mubf, J,A) \ : \ J \in \M(A \times \R^q)^{qd} \text{ such that } \nabla_x \mubf + \nabla_y \cdot J = 0   \text{ in } A \times \R^q \right\},
\end{equation} 
being $\B_W(\mubf, J,A)$ the localized version of $\B_W$, following the notations of Definition~\ref{def:energyBW}:
\begin{equation*}
\B_W(\mubf, J,A) = \int_{A \times \R^q} W(v(x,y)) \, \mubf(\ddr x, \ddr y) + \int_{A \times \R^q}  W'_\infty \left( \frac{\ddr J^\perp}{\ddr |J^\perp|} \right) \, \ddr |J^\perp|.
\end{equation*}

\begin{prop}
\label{prop:TEE_localized_lsc_convex_local}
The function $\TEE$ is convex, l.s.c., local and a measure over $\Mm(\Omega \times \R^q) \times \A(\Omega)$. It satisfies $\T_E(\mubf_u,A) = E(u,A)$ for any $u \in L^0(\Omega,\R^q;m)$ and any $A \in \A(\Omega)$. 
\end{prop}

\begin{proof}
Convexity and lower semi-continuity of $\TEE$ are already proved in Proposition~\ref{prop:TEE_convex_lsc}. The identity $\T_E(\mubf_u,A) = E(u,A)$ comes from Corollary~\ref{crl:lifting_T_dyn}. Locality is immediate given our definition.

Then we claim the following: for $A_1 \subseteq A_2$ open sets of $\Omega$, if $J$ is tangent for $\mubf$ on $A_2$, then its restriction to $A_1$ is tangent for $\TEE(\mubf,A_1)$. Indeed, assume that the restriction of $J$ to $A_1$ is not optimal, that is, $\TEE(\mubf, A_1) < \B_W(\mubf,J,A_1)$. Take $\tilde{J}$ in $\M(A_1 \times \R^q)^{qd}$ which is tangent for $\mubf$ on $A_1$. Then the competitor $ \tilde{J} \1_{A_1} + J \1_{A_2 \setminus A_1}$ still satisfies the generalized continuity equation over $A_2 \times \R^q$ while
\begin{equation*}
\B_W(\mubf, \tilde{J} \1_{A_1} + J \1_{A_2 \setminus A_1}, A_2 ) = \B_W(\mubf, \tilde{J},A_1) +   \B_W(\mubf, J,A_2 \setminus A_1) < \B_W(\mubf,J,A_2),
\end{equation*}
which is a contradiction with $J$ tangent to $\mubf$ on $A_2$. This claim, together with the integral representation of $\B_W$, yields that $\TEE$ is subadditive, superadditive, increasing and inner regular, thus it is a measure~\cite[Theorem 14.23]{dalMaso2012introduction}. 
\end{proof}

\subsection{Counterexamples to the additivity of the Lagrangian lifting}

\paragraph{A heuristic argument}

Before diving into the counterexamples to the additivity of $\T_E$, it is instructive to understand why it fails by trying to prove it. To simplify the reasoning, we assume that $Q$ is concentrated on continuous functions. Following Remark~\ref{rmk:continuous_marginals}, we can therefore introduce the evaluation map $e_x : C(\Omega, \R^q) \to \R^q$ which maps $u$ onto $u(x)$. The condition $Q \in \Pi(\mubf)$ then reads $e_x \# Q = \mubf_x$ for $m$-a.e. $x \in X$, being $(\mubf_x)_{x \in \Omega}$ the disintegration of $\mubf$ with respect to the first variable. More generally, for any set $A \subseteq \Omega$, we can consider the restriction operator $e_A : C(\Omega, \R^q) \to C(A, \R^q)$, so that $e_{\{  x \}}$ coincides with $e_x$.   

Let us take $A_1, A_2$ two disjoint open sets, and for simplicity assume that their closure intersect on a common boundary $\Gamma$. By Proposition~\ref{prop:lifting_TE_properties},  
\begin{equation*}
\T_E(\mubf,A_1) + \T_E(\mubf, A_2) \leq \T_E(\mubf, A_1 \cup A_2)
\end{equation*}
and we would like to prove the other inequality. The natural strategy is to take $Q_1$ optimal for $\mubf$ on $A_1$ and $Q_2$ optimal for $\mubf$ on $A_2$, and try to glue them into a $Q$ which would be optimal for $\mubf$ over $A_1 \cup A_2$. That is, we would like to have $Q \in \Prob(C(\Omega, \R^q))$ such that $e_{A_1} \# Q = Q_1$ and $e_{A_2} = Q_2$. However, this would imply by continuity $e_\Gamma \# Q_1 = e_\Gamma \# Q_2$ as probability distributions over $C(\Gamma, \R^q)$. If $\Gamma = \{ x_0 \}$ is a singleton, then $e_{\{x_0\}} \# Q_1 = e_{\{x_0\}} \# Q_2$ as they both coincide with $\mubf_{x_0}$. However, as soon as $\Gamma$ contains more than two points, there is no guarantee that $e_\Gamma \# Q_1 = e_\Gamma \# Q_2$. 

When $\Omega$ is a segment of $\R$ and $A_1$ and $A_2$ are disjoint intervals then the above reasoning works as their closure intersect at at most one point. In this case $\T_E$ is additive, actually it coincides with $\TEE$  as stated in~\eqref{eq:intro_dynamic_curves}. But even if $\Omega$ is one dimensional, it is enough to take a periodic domain (e.g. the unit circle), so that there exist disjoint intervals $A_1$, $A_2$ whose closure intersect at two distinct points, and additivity may already not hold in this case. Our counterexample precisely tackles this case.

\begin{figure}
\begin{center}
\begin{tabular}{cc}
\begin{tikzpicture}

\draw[line width = 0.5] (0,0) -- (0,6) -- (6,6) -- (6,0) -- cycle ;

\draw (0,0) node[below left]{$0$} ;
\draw (6,0) node[below]{$2 \pi$} ;
\draw (0,3) node[left]{$\pi$} ;
\draw (0,6) node[left]{$2 \pi$} ;

\draw[line width = 1.5pt] (0,0) -- (6,3) ;
\draw[line width = 1.5pt] (0,3) -- (6,6) ;

\draw[dashed, line width = 1pt] (4,0) -- (4,6) ;
\draw (4,0) node[below]{$x$} ;

\filldraw[black] (4,2) circle (2.5pt) node[anchor=north west]{$\sqrt{x}$} ;
\filldraw[black] (4,5) circle (2.5pt) node[anchor=south east]{$-\sqrt{x}$} ;

\end{tikzpicture}

&

\includegraphics[width = 0.5 \textwidth]{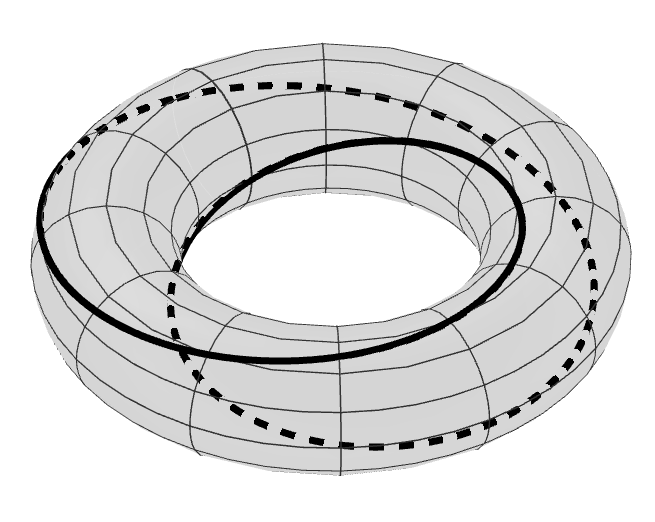}

\end{tabular}
\end{center}
\caption{Two representations of the measure-valued map defined in~\eqref{eq:map_counterexample_circle}, seen as a measure on the product space $\mathbb{S}^1 \times \mathbb{S}^1 \simeq (\R / 2 \pi) \times (\R / 2 \pi)$. Left: representation in the space $\R^2$, with the lateral and vertical sides of the square having to be identified. Right: representation in the space $\R^3$, where $\mathbb{S}^1 \times \mathbb{S}^1$ is embedded as a submanifold. For each interval strictly included in $\R / 2 \pi$, the measure-valued map can be seen as the superposition of two smooth classical maps. However it is not possible to represent it with maps which are continuous over the whole space $\R / 2 \pi$: one would need a $4 \pi$ periodic map instead of a $2 \pi$ periodic one. See Lemma~\ref{lemma:counterexample_circle} for the formalization of this idea.}
\label{fig:counterexample_circle}
\end{figure}
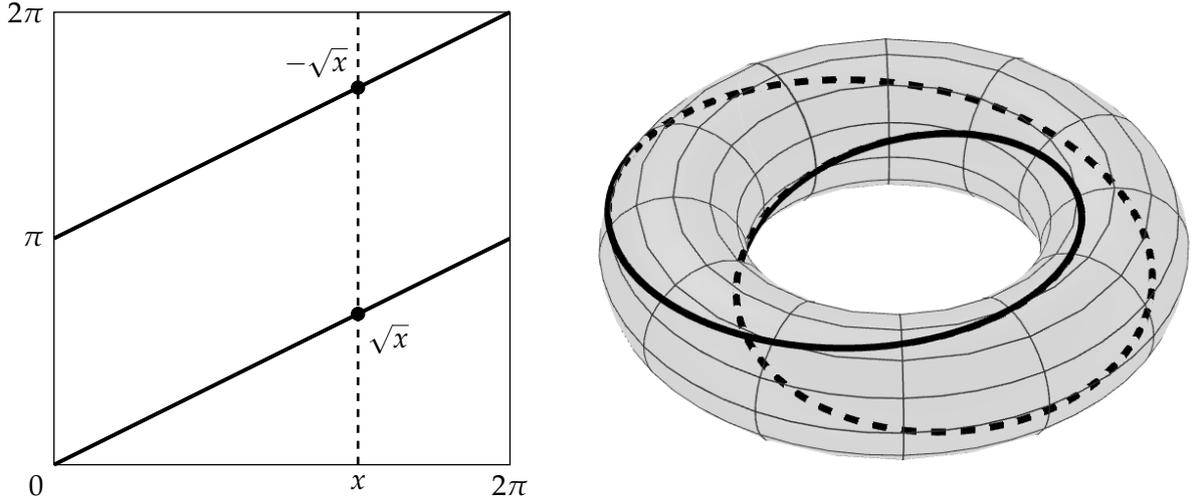

\paragraph{Maps defined on the unit circle}

Let $\mathbb{S}^1$ be the unit circle of $\R^2$, that we see as a subset of $\mathbb{C}$. We take $X = Y = \mathbb{S}^1$, and we endow $\mathbb{S}^1$ with the Lebesgue measure $m$ (a.k.a. Haar measure in this case). The functional $E$ we consider is the classical action of a curve, that is,
\begin{equation*}
E(u,A) = \begin{cases}
\dst{\frac{1}{2} \int_A |\dot{u}_t|^2 \, \ddr t}  & \text{if } u \in W^{1,2}(A, \mathbb{S}^1), \\
+ \infty & \text{otherwise}.
\end{cases}
\end{equation*}
We consider $\T_E$ the Lagrangian lifting of the functional to $\Mm(\mathbb{S}^1 \times \mathbb{S}^1) \times \A(\mathbb{S}^1)$. As curve of measures, we take $\mubf : \mathbb{S}^1 \to \Prob(\mathbb{S}^1)$ given as
\begin{equation}
\label{eq:map_counterexample_circle}
\mubf : x \mapsto \mubf_x = \frac{1}{2} (\delta_{\sqrt{x}} + \delta_{- \sqrt{x}}),
\end{equation} 
where $\sqrt{x}$ stands for a complex square root of $x$, see Figure~\ref{fig:counterexample_circle}. Equivalently we can see it as a measure on the product space $\mathbb{S}^1 \times \mathbb{S}^1$ whose first marginal is $m$ and whose disintegration is $(\mubf_x)_{x \in \mathbb{S}^1}$.

\begin{lm}
\label{lemma:counterexample_circle}
If $A$ is an open set of $\mathbb{S}^1$ which is not dense in $\mathbb{S}^1$ then $\T_E(\mubf,A) \leq m(A)/8$, while on the other hand $\T_E(\mubf,\mathbb{S}^1) = + \infty$.
\end{lm}

\noindent Clearly this lemma shows that $\T_E$ is not subadditive as we can easily find $A_1, A_2$ both not dense in $\mathbb{S}^1$ such that $\mathbb{S}^1 = A_1 \cup A_2$.

\begin{proof}
Take $A$ an open set not dense in $\mathbb{S}^1$, say which does not contain a neighborhood of $\exp(i t_0)$. For $t \in [t_0, t_0 + 2 \pi)$ such that $\exp(it) \in A$, we define $u_1(\exp(it)) = \exp(it/2)$ and $u_2(\exp(it)) = -\exp(it/2)$. This defines $u_1, u_2$ two smooth functions over $A$. We set $\tilde{Q} = (\delta_{u_1} + \delta_{u_2}) /2$ as a probability measure on $L^0(A,\R^q;m)$ and we use it in~\eqref{eq:alternate_representation_lifting_TE} to estimate $\T_E(u,A)$. We have $|\dot{u}_i(t)| = 1/2$ for $i=1,2$ on $A$, thus we obtain the estimate $\T_E(u,A) \leq m(A)/8$. 

On the other hand, if we look at the whole circle $\mathbb{S}^1$ and if $Q \in \Pi(\mubf)$ is such that $\int E(u, \mathbb{S}^1) \, Q(\ddr u) < + \infty$, then $Q$-a.e. $u$ is such that $u(x)^2 = x$ for a.e. $x \in \mathbb{S}^1$. That is, $Q$-a.e. $u$ is a $W^{1,2}(\mathbb{S}^1, \mathbb{S}^1)$, thus continuous, selection of the complex square root. This is not possible, which implies that there is no $Q \in \Pi(\mubf)$ with $\int E(u,\mathbb{S}^1) \, Q(\ddr u) < + \infty$, hence $\T_E(u,\mathbb{S}^1) = + \infty$.    
\end{proof}

\paragraph{Maps defined on the unit disk}

The example can be extended to maps defined on the unit disk. Still we identify $\R^2$ with the complex plane $\mathbb{C}$. Calling $\mathbb{B}^2$ the unit disk of $\R^2$ and $m$ the Lebesgue measure on it, we consider the Dirichlet energy $E$ defined on $L^0(\mathbb{B}^2,\mathbb{B}^2;m) \times \A(\Omega)$ by 
\begin{equation*}
E(u,A) = \begin{cases}
\dst{\frac{1}{2} \int_A |\nabla u|^2} & \text{if } u \in W^{1,2}(A,\mathbb{B}^2), \\
+ \infty & \text{otherwise}. 
\end{cases}
\end{equation*}
We also look at $\T_E$ the lift of the functional $E$ onto $\Mm(\mathbb{B}^2 \times \mathbb{B}^2) \times \A(\mathbb{B}^2)$. Given a point $x$ that we write in polar coordinates $x = r \exp(it)$, we look at the map of probability measure
\begin{equation*}
\mubf : r \exp(it) \mapsto \mubf_x = \frac{1}{2} (\delta_{r \exp(it/2)} + \delta_{-r \exp(it/2)}),
\end{equation*}
and we see it as a measure on $\mathbb{B}^2 \times \mathbb{B}^2$ whose first marginal is $m$ and whose disintegration is given by $(\mubf_x)_{x \in \mathbb{B}^2}$. Similarly to the previous case, we can prove the following. 

\begin{lm}
Let $A$ an open set of $\mathbb{B}^2$ which does not contain a neighborhood of the half line $\{ r \exp(it_0) \, : \,  r \in [0,1] \}$ for some $t_0$, then $\T_E(\mubf,A) \leq 5 m(A)/8$, while on the other hand if $A$ is any corona centered at $0$, $\T_E(\mubf,A) = + \infty$. 
\end{lm}

\begin{proof}
For the first claim, for $t \in [t_0, t_0 + 2 \pi)$ and $r \geq 0$ such that $r \exp(it) \in A$ we can define $u_1$, $u_2$ by $u_1(r \exp(it)) = r \exp(it/2)$ and $u_2(r \exp(it)) = -r \exp(it/2)$. Similarly to the proof of Lemma~\ref{lemma:counterexample_circle}, we take $\tilde{Q} = (\delta_{u_1} + \delta_{u_2}) /2$ as a competitor in expression~\eqref{eq:alternate_representation_lifting_TE}. As $|\nabla u_i|^2 \leq 5/4$ on $A$ for $i=1,2$, we get our estimate $\T_E(\mubf,A) \leq 5 m(A)/8$. 

On the other hand,  take a corona $A$, in particular it contains a circle $C = \{ x \ : \ |x| = r \}$ centered at the origin. If $Q \in \Pi(\mubf)$ with $\int E(u,A) \, Q(\ddr u) < + \infty$, then $Q$-a.e. $u$ is such that $u(x)^2 =   |x| x$ for a.e. $x \in A$. As $Q$-a.e. $u$ is a $W^{1,2}(A,\mathbb{B}^2)$ function, we can consider the trace of $u$ on $C$, and we obtain that $Q$-a.e. function $u$ has a trace on $C$ which is (up to rescaling) a $W^{1/2,2}(C,C)$ selection of the complex square root. Although $W^{1/2,2}$ functions are not continuous in one dimension, it is still impossible for them to be a selection of the complex square root as we already showed in our previous work~\cite[Lemma 5.4]{Lavenant2019}.  
\end{proof}

\paragraph{A smoothing argument}

In the counterexamples above the measures $\mubf$ are very singular, as they are valued in the set of finite combinations of Dirac masses. However, we can still prove that additivity is lost even when the measures are smoothed. Indeed, suppose that for some $\mubf \in \Mm(\Omega \times \R^q)$ and some open sets $A_1$, $A_2$ in $\A(\Omega)$ we have 
\begin{equation}
\label{eq:violation_additivity}
\T_E(\mubf, A_1) + \T_E(\mubf,A_2) < \T_E(\mubf, A_1 \cup A_2), 
\end{equation}
with the right hand side being potentially $+ \infty$.
To smooth out, we rely on the following lemma that we state in its un-localized version. 

\begin{lm}
\label{lemma:smoothing_TE}
Assume $\Omega$ is a subset of $\R^d$ endowed with its Lebesgue measure $m$, and consider a functional $E : L^0(\Omega, \R^q;m) \to [0, + \infty]$ which is translation invariant, in the sense $E(u + z) = E(u)$ for any $u \in L^0(\Omega,\R^q;m)$ and any fixed $z \in \R^q$. Then, for any non-negative mollifier $\chi$ on $\R^q$, and any $\mubf \in \Mm(\Omega \times \R^q)$,
\begin{equation*}
\T_E(\chi *_y \mubf) \leq \T_E(\mubf).
\end{equation*} 
Here $\chi *_y \mubf \in \Mm(\Omega \times \R^q)$ is defined via its disintegration $((\chi *_y \mubf)_{x})_{x \in X}$ with respect to the first coordinate: for $m$-a.e. $x \in \Omega$, the measure $(\chi *_y \mubf)_{x}$ is the convolution of $\chi$ with $\mubf_x$.
\end{lm}

\begin{proof}
If $Q \in \Pi(\mubf)$, and if $z \in \R^q$, we define $Q_z$ as the push forward of $Q$ by the map $u \mapsto (u-z)$, and we set
\begin{equation*}
\tilde{Q} = \int_{\R^q} Q_z \, \chi(z) \, \ddr z.
\end{equation*} 
We can check that $\tilde{Q} \in \Pi(\chi *_y \mubf)$: for any test function $\varphi : \Omega \times \R^q \to \R$, using Fubini's theorem,
\begin{align*}
\int_{L^0(\Omega, \R^q;m)}\left( \int_\Omega \varphi(x,u(x)) \, \ddr x \right) \, \tilde{Q}(\ddr u)  
& = \int_{L^0(\Omega, \R^q;m)} \left( \int_{\R^q} \left( \int_\Omega \varphi(x,u(x)-z) \, \ddr x \right) \, \chi(z) \, \ddr z \right) \, Q(\ddr u) \\
& = \int_{L^0(\Omega, \R^q;m)}  \left( \int_\Omega (\varphi *_y \chi)(x,u(x)) \, \ddr x  \right) \, Q(\ddr u) \\
& = \int_{\Omega \times \R^q} (\varphi *_y \chi)(x,y) \, \mubf(\ddr x, \ddr y).
\end{align*}
Moreover $\int E \, \ddr \tilde{Q} = \int E \, \ddr Q$ as $E$ is invariant by translation. The conclusion follows. 
\end{proof}

If we take $(\chi_n)_{n \geq 1}$ a sequence of mollifiers which is an approximation of unity, it is clear that $\chi_n *_y \mubf$ converges narrowly to $\mubf$ as $n \to + \infty$. By lower semi-continuity of $\T_E$ and Lemma~\ref{lemma:smoothing_TE},
\begin{equation*}
\lim_{n \to + \infty} \T_E(\chi_n *_y \mubf) = \T_E (\mubf). 
\end{equation*}

Going back to the localized version and to the notations of~\eqref{eq:violation_additivity}, applying this reasoning, we deduce that for $n$ large enough additivity still fails:
\begin{equation*}
\T_E(\chi_n *_y \mubf, A_1) + \T_E(\chi_n *_y \mubf,A_2) < \T_E(\chi_n *_y \mubf, A_1 \cup A_2).
\end{equation*}
The difference with~\eqref{eq:violation_additivity} is that $\chi_n *_y \mubf$ is arguably ``smoother'' than $\mubf$: it always has a density with respect to the Lebesgue measure on $\Omega \times \R^q$, and this density can be made smooth if $(\chi_n)_{n \geq 1}$ is smooth.

\subsection{Characterization of the Eulerian lifting}

We arrive to the main result of this section: we characterize $\TEE$ as the convex, l.s.c. and \emph{subadditive} envelope of $\widetilde{\T}_E$.  As a preliminary result we first prove in the general case the existence of the convex, l.s.c. and \emph{subadditive} envelope of a localized functional $E$. Similarly to~\eqref{eq:defi_tilde_T} we define 
\begin{equation*}
\widetilde{\T}_E(\mubf,A) = \begin{cases}
E(u,A) & \text{if } \mubf = \mubf_u \text{ for some } u \in L^0(X,Y;m), \\
+ \infty & \text{otherwise}.
\end{cases}
\end{equation*}

\begin{prop}
For any $E : L^0(\Omega,\R^q;m) \times \A(\Omega) \to [0,+\infty]$, there exists a unique subadditive, increasing, inner regular, convex and l.s.c. functional which is the largest among all subadditive, increasing, inner regular, convex l.s.c. functionals $\mathcal{T} : \Mm(\Omega \times \R^q) \times \A(\Omega) \to [0, + \infty]$ such that $\mathcal{T} \leq \widetilde{\T}_E$. We denote this functional by $\overline{\mathcal{T}}_E$.
\end{prop}

\noindent Note that by definition we always have $\overline{\mathcal{T}}_E \leq \mathcal{T}_E \leq \widetilde{\T}_E$.

\begin{proof}
We first define 
\begin{equation*}
\widehat{\T}_E = \sup_{\T} \, \T, 
\end{equation*}
where the supremum is taken over $\T$ increasing, subadditive, convex, l.s.c. (but not necessarily inner regular) and such that $\T \leq \widetilde{\T}_E$. Using that increasing, convex, l.s.c and subadditive functionals are stable by (uncountable) supremum, we see that $\widehat{\T}_E$ is increasing, convex, l.s.c and subadditive (but a priori not inner regular) and $\widehat{\T} \leq \widetilde{\T}_E$. 
Then we define $\overline{\T}_E \leq \widehat{\T}_E \leq \widetilde{\T}_E$ as the inner regular envelope of $\widehat{\T}_E$ following~\cite[Definition 14.4]{dalMaso2012introduction}, that is,      
\begin{equation*}
\overline{\T}_E(\mubf,A) = \sup_{B} \, \widehat{\T}_E(\mubf,B), 
\end{equation*}
where the supremum is taken over $B \in \A(\Omega)$ compactly included in $A$. The functional $\overline{\T}_E$ is still convex, l.s.c. (these properties are preserved by supremum), increasing, and moreover it is subadditive~\cite[Proposition 14.19]{dalMaso2012introduction}.

Eventually, let $\mathcal{T} : \Mm(\Omega \times \R^q) \times \A(\Omega) \to [0, + \infty]$ such that $\mathcal{T} \leq \widetilde{\T}_E$ and which is subadditive, increasing, inner regular, convex and l.s.c. Given how $\widehat{\T}_E$ is constructed, $\mathcal{T} \leq \widehat{\T}_E$. Thus by inner regularity of $\T$ and given how $\overline{\T}_E(\mubf,A)$ is defined, we have $\T(\mubf,A) \leq \overline{\T}_E(\mubf,A)$ for any $\mubf$ and any $A$. That proves the maximality of $\overline{\T}_E$.
\end{proof}

Our main result will hold for functionals $E$ of the form $u \mapsto \int W(\nabla u)$, but we will need a slightly stronger assumption than Assumption~\ref{asmp:W}. This additional technical condition will be used only in Lemma~\ref{lm:regularization_measure_valued}.

\begin{asmp}
\label{asmp:W_stronger}
The function $W : \R^{qd} \to [0, + \infty)$ is convex. Moreover it is approximately radial and growing at least like $|v|^p$ for $p \in [1, + \infty)$ in the sense
\begin{equation*}
C_1 f(|v|) - C_2 \leq W(v) \leq C_3 f(|v|) + C_4,
\end{equation*}
for a non-decreasing function $f : [0, + \infty) \to [0,+\infty)$ which satisfies $f(r) \geq r^p$, some constants $C_3 \geq C_1 > 0$ and $C_2, C_4 \geq 0$. Here, $| \cdot |$ denotes any norm on the space of matrices. 
\end{asmp}

\begin{theo}[characterization of $\TEE$]
\label{theo:main}
For $\Omega$ a bounded open set of $\R^d$ and $W$ which satisfies Assumption~\ref{asmp:W_stronger}, we define the localized functional $E : L^0(\Omega, \R^q;m) \times \A(\Omega) \to [0, + \infty]$ as the localized version of Definition~\ref{defi:E_from_W}, see Example~\ref{ex:localization_E}. Then
\begin{equation*}
\overline{\mathcal{T}}_E = \TEE,
\end{equation*}
that is, the localized version of the Eulerian lifting (see~\eqref{eq:TEE_localized}) is the subadditive, increasing, inner regular, convex and l.s.c. envelope of $E$ on measure-valued maps.
\end{theo}

\noindent In particular this envelope $\overline{\mathcal{T}}_E$ is actually local and a measure by Proposition~\ref{prop:TEE_localized_lsc_convex_local}. Note that even though the theorem is stated here for a functional depending only on $\nabla u$, it is easy to add \emph{any} continuous zero order term as they are linear, so that we deal with functionals like~\eqref{eq:intro_energy_generic} as presented in the introduction.

\begin{crl}
Let $\Omega$, $W$ and $E$ as in Theorem~\ref{theo:main}. For a continuous and bounded function $f : \Omega \times \R^q \to \R$ we define $E_L : (u,A) \mapsto \int_A f(x,u(x)) \, \ddr x$, and following Proposition~\ref{prop:envelope_linear} we define
\begin{equation*}
\T_{E_L}(\mubf,A) = \int_{A \times \R^q} f(x,y) \, \mubf(\ddr x, \ddr y).
\end{equation*}

Then $\overline{\mathcal{T}}_{E+E_L}$, the subadditive, increasing, inner regular, convex and l.s.c. envelope of $E + E_L$ on measure-valued maps, is equal to $\TEE + \T_{E_L}$.
\end{crl}

\begin{proof}
This is straightforward once we notice that, $\T_{E_L}$ being continuous, local and a measure, the functional $\overline{\mathcal{T}}_{E+E_L} - \T_{E_L}$ should be the subadditive, increasing, inner regular, convex and l.s.c. envelope of $E$, and thus we can apply Theorem~\ref{theo:main}.
\end{proof}

Before moving on to the proof of Theorem~\ref{theo:main}, let us give a heuristic argument to explain why it is plausible. Let us fix $\mubf \in \Mm(\Omega \times \R^q)$, and assume for simplicity that $W$ is strictly convex and that the (unique) tangent $J$ has density $v : \Omega \times \R^q \to \R^{qd}$ with respect to $\mubf$. We also take $Q \in \Prob(L^0(\Omega, \R^q;m))$ optimal for $\T_E$. We claim that, if we had $\T_E(\mubf, \Omega)= \TEE(\mubf, \Omega)$, then, for $Q$-a.e. $u$, there holds 
\begin{equation}
\label{eq:nablau_v}
\text{for } m \text{-a.e. } x \in \Omega, \quad  \nabla u(x) = v(x, u(x)).
\end{equation}
Indeed, consider $\tilde{J} = \int J_u \, Q(\ddr u) \in \M(\Omega \times \R^q)^{qd}$ where $J_u$, tangent to $\mubf_u$, is defined in the proof of Corollary~\ref{crl:lifting_T_dyn}. As $\nabla_x \mubf_u + \nabla_y \cdot J_u = 0$, averaging w.r.t. $u$ we find that $\nabla_x \mubf + \nabla_y \cdot \tilde{J} = 0$. Using Equation~\eqref{eq:dual_BW} to represent $\B_W$,
\begin{align*}
\TEE(\mubf, \Omega) \leq \B_W(  \mubf,  \tilde{J}, \Omega) & = \sup_{b} \left\{ \int_{\Omega \times \R^q} b \, \ddr \tilde{J} - \int_{\Omega \times \R^q} W^*(b) \, \ddr \mubf \right\} \\
& = \sup_{b} \left\{ \int_{L^0(\Omega, \R^q;m)} \left( \int_{\Omega} [ b(x,u(x)) \cdot \nabla u(x)  -W^*(b(x,u(x))) ] \, \ddr x \right) \, Q(\ddr u) \right\} \\
& \leq  \int_{L^0(\Omega, \R^q;m)} \left( \int_{\Omega} W(\nabla u(x)) \, \ddr x \right) \, Q(\ddr u) \\
& = \int_{L^0(\Omega, \R^q;m)} E(u,\Omega) \, Q(\ddr u) = \T_E(u, \Omega) 
\end{align*}
where the supremum is taken over all functions $b : \Omega \times \R^q \to \R^{qd}$, and the second inequality comes from the definition of $W^*$. Thus we see that $\TEE(\mubf) \leq \T_E(u)$ (that we already knew), and if there is equality then the inequalities are equalities, so that $\tilde{J} = J$ is tangent and that $Q$-a.e. $u$ satisfies $\nabla u(x) = (\nabla W^*)(b(x,u(x)))$ for the optimal $b$. This implies that the density of $\tilde{J}$ with respect to $\mubf$ is $(\nabla W^*)(b)$, which coincides with $v$ as $\tilde{J} = J$. Thus we get that $Q$-a.e. $u$ satisfies $\nabla u(x) = (\nabla W^*)(b(x,u(x))) = v(x,u(x))$ for $m$-a.e. $x$, that is, \eqref{eq:nablau_v}.

If $\Omega$ is a segment of $\R$ then~\eqref{eq:nablau_v} is an ODE, it can be solved (regularity issues aside): this is precisely how the identity~\eqref{eq:intro_dynamic_curves} is proved, see~\cite[Theorem 8.2.1]{AGS}. When $\Omega$ is of dimension more than $2$ then~\eqref{eq:nablau_v} is a much stronger constraint and it implies conditions on $v$ which have no reason to be satisfied as we already explained in our previous work~\cite[Section 5.3]{Lavenant2019}.

On the other hand, if $\Omega$ is small enough, then there is a hope that we can solve~\eqref{eq:nablau_v} approximately: this is what is done in the proof below by projecting~\eqref{eq:nablau_v} on rays emanating from a point in $\Omega$, see the estimate~\eqref{eq:localization}. If $\Omega$ is not small, we cut it in small pieces, use the approximate version~\eqref{eq:localization} on each piece, and then patch the results. We can patch the results precisely because $\overline{\T}_E$ is sudadditive, it could not be possible with $\T_E$.

\begin{proof}[\textbf{Proof of Theorem~\ref{theo:main}}]
As $\TEE$ is convex, l.s.c, a measure, and satisfies $\TEE \leq \widetilde{\T}_E$ (see Proposition~\ref{prop:TEE_localized_lsc_convex_local}), we only need to show that
\begin{equation*}
\overline{\mathcal{T}}_E(\mubf,A) \leq \TEE(\mubf,A).
\end{equation*}
for any $\mubf \in \Mm(\Omega \times \R^q)$ and any open set $A \in \A(\Omega)$. We fix such $\mubf$, $A$ until the end of the proof and without loss of generality we can assume $\TEE(\mubf,A) < + \infty$.

\medskip

\textit{First step. Regularization}. This step is rather classical and we will refer Lemma~\ref{lm:regularization_measure_valued} that we put in the appendix for to the precise statement and the proof. With the help of this lemma, given $\varepsilon > 0$, $\tilde{A}$ an open set compactly embedded in $A$ and $\mathcal{V}$ a neighborhood of the restriction of $\mubf$ to $\tilde{A} \times \R^q$ for the topology of narrow convergence, we find $\tilde{\mubf} \in \mathcal{V}$ which has a density $\tilde{\rho}$ with respect to the Lebesgue measure. This density $\tilde{\rho}$ is supported on $\tilde{A} \times B$ and is smooth and bounded from below by a strictly positive constant on $\tilde{A} \times B$, being $B$ a ball in $\R^q$. Moreover, there is also $\tilde{v} \in C^\infty(\tilde{A} \times B, \R^{qd})$ such that $(\tilde{\mubf}, \tilde{v} \tilde{\mubf})$ satisfy~\eqref{eq:continuity_smooth} the generalized continuity equation as well as the estimate~\eqref{eq:lm_reg_estimate_energy} which reads $\B_W(\tilde{\mubf}, \tilde{v} \tilde{\mubf},\tilde{A}) \leq \TEE(\mubf,A) + \varepsilon$. In the next two steps we will prove that 
\begin{equation}
\label{eq:to_prove_after_reg}
\overline{\mathcal{T}}_E(\tilde{\mubf},\tilde{A}) \leq \B_W(\tilde{\mubf}, \tilde{v} \tilde{\mubf}, \tilde{A}).
\end{equation}
It will be enough to prove the theorem. Indeed, from~\eqref{eq:to_prove_after_reg},we obtain $\overline{\mathcal{T}}_E(\tilde{\mubf},\tilde{A}) \leq \TEE(\mubf,A) + \varepsilon$ from~\eqref{eq:lm_reg_estimate_energy}, thus we can take $\varepsilon > 0$ going to $0$ and the neighborhood $\mathcal{V}$ shrinking to $\{ \mubf \}$ to obtain $\overline{\mathcal{T}}_E(\mubf,\tilde{A}) \leq \TEE(\mubf,A)$ thanks to the lower semi-continuity of $\overline{\T}_E(\cdot, \tilde{A})$. Then we take the supremum over $\tilde{A} \subset A$ with compact inclusion in $A$ and get $\overline{\mathcal{T}}_E(\mubf,A) \leq \TEE(\mubf,A)$ thanks to the inner regularity of the envelope $\overline{\T}_E$, which is our conclusion. Thus, we are now left with the proof of~\eqref{eq:to_prove_after_reg}.

\medskip

\textit{Second step. Approximation on small sets}. Now we fix $\tilde{\mubf}$, $\tilde{\rho}, \tilde{v}$ and $\tilde{A}$ as given in the first step, and we write $\tilde{A} \times B$ for the support of $\tilde{\rho}$. We denote by $\lesssim$ an equality which holds up to a multiplicative constant which depends only on $\tilde{v}$ as well as its derivatives. We consider $\tilde{\mubf}_x = \tilde{\rho}(x,y) \ddr y$ for any $x \in \tilde{A}$: the family $(\tilde{\mubf}_x)_{x \in \tilde{A}}$ is the disintegration of $\tilde{\mubf}$ with respect to its variable.

We will prove that if $\hat{A}$ is a starshaped domain included in $\tilde{A}$, then we can find $Q \in \Pi(\tilde{\mubf})$ such that
\begin{equation}
\label{eq:localization}
\frac{1}{m(\hat{A})} \left| \int_{L^0(\hat{A},\R^q;m)} E(u,\hat{A}) \, Q(\ddr u)  -  \B_W(\tilde{\mubf}, \tilde{v} \tilde{\mubf}, \hat{A}) \right| \lesssim \mathrm{diameter}(\hat{A}).
\end{equation}
To that end, fix $x_0 \in \hat{A}$ such that the segment connecting $x_0$ to any $x \in \hat{A}$ stays in $\hat{A}$. For a fixed $x \in \hat{A}$, we call $\tilde{\rho}^x_t(y) = \tilde{\rho}((1-t) x_0 + t x, y)$ for $t \in [0,1]$: this is the density on the segment joining $x_0$ to $x$. Projecting the (smooth) continuity equation~\eqref{eq:continuity_smooth} on the line $(x- x_0)$, we see that this curve of measures solve the classical continuity equation
\begin{equation}
\label{eq:continuity_equation_classical}
\begin{cases}
\dr_t \tilde{\rho}^x_t + \nabla_y \cdot( w^x \tilde{\rho}^x_t) = 0 & \text{in } [0,1] \times B, \\
w^x \cdot \mathbf{n}_B = 0 & \text{on } [0,1] \times \dr B
\end{cases}
\qquad \text{with} \qquad w^x(t,y) = \tilde{v}((1-t)x_0+t x,y) (x - x_0),
\end{equation} 
where $\tilde{v}((1-t)x_0+t x,y) (x - x_0)$ is to be read as a matrix vector product and $\mathbf{n}_B$ is the outward normal to $B$.
For this equation we can apply the classical superposition principle, see~\cite[Proposition 8.1.8]{AGS} or~\cite[Theorem 4.4]{SantambrogioOTAM}. Specifically for any $x \in \hat{A}$, we call $\Psi(x,y)$ the solution at time $t=1$ of the ODE starting at $x_0$ at time $0$ which reads 
\begin{equation*}
\dot{y}_t = w^x(t,y_t) = \tilde{v}((1-t)x_0+tx, y_t ) (x - x_0)
\end{equation*}
It defines a function $\Psi : \hat{A} \times B \to B$ which is smooth (of class $C^\infty$) by classical regularity for flows of ODEs. From the classical superposition principle~\cite[Proposition 8.1.8]{AGS} applied to~\eqref{eq:continuity_equation_classical}, there holds $\Psi(x, \cdot) \# \tilde{\mubf}_{x_0} = \tilde{\mubf}_x$ for any $x \in \hat{A}$. Moreover, as $w^x(t,y) = v(x_0,y) (x-x_0) + o(x-x_0)$ uniformly in $t \in [0,1]$, that is, $w^{x_0}(t,y) = 0$ and $\nabla_x w^x(t,y)|_{x=x_0} = v(x_0,y)$, standard theory of ODEs yields that we can differentiate the flow and obtain: 
\begin{equation*}
\Psi(x_0,\cdot) = \mathrm{Id}, \qquad \nabla_x \Psi(x_0, \cdot)  = \tilde{v}(x_0, \cdot ) .
\end{equation*}

To produce a competitor in~\eqref{eq:to_prove_after_reg} we follow the idea of a ``parametric'' $Q$ developed in Example~\ref{ex:parametric_Q}, where now $y \in B$ is the ``parameter''. Indeed for every $y$, we can look at $u_y : x \mapsto \Psi(x,y)$, and we simply set $Q$ as the pushforward of $\tilde{\mubf}_{x_0} = \tilde{\rho}(x_0,y) \ddr y$ by the map $y \mapsto u_y$. In particular $Q$ is concentrated on $C^\infty(\hat{A},B)$ and as $\Psi(x,\cdot) \# \tilde{\mubf}_{x_0} = \tilde{\mubf}_x$, the computations of Example~\ref{ex:parametric_Q} guarantee $Q \in \Pi(\tilde{\mubf})$. On the other hand for the value of the functional following Example~\ref{ex:parametric_Q_E} we find 
\begin{equation*}
\int_{L^0(\hat{A},\R^q;m)} E(u,\hat{A}) \, Q(\ddr u) = \int_{B} \left( \int_{\hat{A}} W(\nabla_x \Psi(x,y)) \, \ddr x \right) \, \tilde{\rho}(x_0, y) \ddr y.
\end{equation*}
To write the other term in the left hand side of~\eqref{eq:localization}, we use $\Psi(x, \cdot) \# \tilde{\mubf}_{x_0} = \tilde{\mubf}_x = \tilde{\rho}(x,y) \ddr y$ and obtain  
\begin{align*}
\B_W(\tilde{\mubf}, \tilde{v} \tilde{\mubf}, \hat{A}) & = \int_{\hat{A} \times B} W(\tilde{v}(x,y))  \tilde{\rho}(x, y) \, \ddr x \ddr y \\
& = \int_{\hat{A}} \left( \int_B W(\tilde{v}(x,y)) \, \tilde{\rho}(x, y) \ddr y \right) \ddr x \\
& = \int_{\hat{A}} \left( \int_{B} W(\tilde{v}(x, \Psi(x,y))) \,  \tilde{\rho}(x_0, y) \ddr y  \right) \, \ddr x \\
& = \int_{B} \left( \int_{\hat{A}} W(\tilde{v}(x, \Psi(x,y))) \, \ddr x   \right) \,  \tilde{\rho}(x_0, y) \ddr y
\end{align*}
where the last equality is Fubini's theorem. The two terms in the left hand side of~\eqref{eq:localization} now look alike:
\begin{multline*}
\left| \int_{L^0(\hat{A},\R^q;m)} E(u, \hat{A}) \, Q(\ddr u)  -  \B_W(\tilde{\mubf}, \tilde{v} \tilde{\mubf}, \hat{A})   \right| \\
 = \left| \int_{B} \left( \int_{\hat{A}} W(\nabla_x \Psi(x,y)) - W(\tilde{v}(x, \Psi(x,y))) \, \ddr x \right) \, \tilde{\rho}(x_0, y) \, \ddr y \right|.
\end{multline*}
We use smoothness of $\Psi$ and $\tilde{v}$ to conclude: as $\Psi$ is a $C^\infty$ function on $\hat{A} \times B$, with derivatives controlled only by the ones of $\tilde{v}$, and as $\nabla_x \Psi(x_0, y) = \tilde{v}(x_0,y) = \tilde{v}(x_0, \Psi(x_0,y) )$ for any $y$, we obtain uniformly in $y \in B$
\begin{equation*}
| \nabla_x \Psi(x, y)  - \tilde{v}(x, \Psi(x,y) ) | \lesssim |x - x_0| \leq \mathrm{diameter}(\hat{A}).
\end{equation*}
Thanks to $W$ convex and finite everywhere thus Lipschitz on bounded sets, we conclude to 
\begin{align*}
& \left| \int_{B} \left( \int_{\hat{A}} W(\nabla_x \Psi(x,y)) - W(\tilde{v}(x, \Psi(x,y))) \, \ddr x \right) \, \tilde{\rho}(x_0, y) \, \ddr y \right| \\
& \qquad \qquad \qquad \leq \mathrm{Lip}(W) \int_{B} \left( \int_{\hat{A}} | \nabla_x \Psi(x, y)  - \tilde{v}(x, \Psi(x,y) ) | \, \ddr x   \right) \, \tilde{\rho}(x_0, y) \, \ddr y \\
& \qquad \qquad \qquad  \lesssim \mathrm{Lip}(W) m(\hat{A}) \, \mathrm{diameter}(\hat{A}),
\end{align*}
being $\mathrm{Lip}(W)$ a Lipschitz constant of $W$ on the image of $\tilde{v}$ and $\nabla_x \Psi$. This implies our estimate~\eqref{eq:localization} which we rewrite
\begin{equation*}
\overline{\mathcal{T}}_E(\mubf,\hat{A}) \leq \T_E(\mubf,\hat{A})  \leq \B_W(\tilde{\mubf}, \tilde{v} \tilde{\mubf}, \hat{A})  + C m(\hat{A}) \mathrm{diam}(\hat{A}),
\end{equation*}
where $C$ depends only on $W$ and $\tilde{v}$ as well as its derivatives.

\medskip

\textit{Third step. Conclusion.} Fix $\delta > 0$. We take a partition (up to a negligible set) of $\tilde{A}$ into $n$ pieces $\hat{A}_1, \ldots, \hat{A}_n$, each of them starshaped and of diameter at most $\delta$. As $\overline{\mathcal{T}}_E$ is subadditive but $\B_W$ is additive, we obtain
\begin{equation*}
\overline{\mathcal{T}}_E(\mubf,\tilde{A}) \leq \sum_{i=1}^n \overline{\mathcal{T}}_E(\mubf,\hat{A}_i) \leq \sum_{i=1}^n \left( \B_W(\tilde{\mubf}, \tilde{v} \tilde{\mubf}, \hat{A}_i) + C \delta m(\hat{A}_i) \right) 
 = \B_W(\tilde{\mubf}, \tilde{v} \tilde{\mubf}, \tilde{A}) + C \delta m(\tilde{A}).
\end{equation*} 
As we take more and more pieces and send $\delta \to 0$, we get exactly $\overline{\mathcal{T}}_E(\mubf,\tilde{A}) \leq \B_W(\tilde{\mubf}, \tilde{v} \tilde{\mubf}, \tilde{A})$, that is,~\eqref{eq:to_prove_after_reg}.
\end{proof}

\appendix

\section{Some technical results}

In this appendix we collect some results which are quite classical and/or use techniques which are not central in the present work. 

\subsection{Checking the assumption of coercivity of the functional \texorpdfstring{$E$}{E}}
\label{sec:appendix_assmp_A}

We present a setting tailored for the applications in the rest of this article in which Assumption~\ref{asmp:reg_XYm_E} is satisfied. It illustrates that checking Assumption~\ref{asmp:reg_XYm_E} is more or less equivalent to checking existence of a minimizer for a problem of calculus of variations for classical maps, thus techniques are by now well-established. 

\begin{prop}
\label{prop:checking_assumption_A}
We fix $X = \Omega$ a bounded open set of $\R^d$ that we endow with $m$ the Lebesgue measure. On the other hand the space $Y$ is taken to be $Y = \R^q$.  For an integer $k \geq 1$ and $p \in (1, + \infty)$, let $L = L(x,u,v_1, v_2, \ldots, v_k) \in (0, + \infty)$ a Lagrangian which is Carathéodory, continuous in its last $k+1$ variables for fixed $x \in \Omega$ and  and convex in its last variable when the other ones are fixed. We also assume that it grows at least like $|v_k|^p$ in its last variable, that is
\begin{equation*}
L(x,u,v_1, v_2, \ldots, v_k) \geq C_1 |v_k|^p - C_2
\end{equation*}
for some constants $C_1> 0$, $C_2 \geq 0$, uniformly in the variables $x,u,v_1, \ldots, v_{k-1}$. We consider the functional $E$ defined over $L^0(\Omega, \R^q;m)$ by 
\begin{equation*}
E(u) = \begin{cases}
\dst{\int_\Omega L(x, u(x), \nabla u (x), \nabla^2 u(x), \ldots, \nabla^k u(x)) \, \ddr x} & \text{if } u \in W^{k,p}(\Omega, \R^q),  \\
 + \infty & \text{otherwise},
\end{cases}
\end{equation*}
being $W^{k,p}(\Omega,\R^q)$ the standard Sobolev space of functions having $k$ derivatives in $L^p$. 
Then the functional $E$ satisfies Assumption~\ref{asmp:reg_XYm_E}. 
\end{prop}

\begin{proof}
Using \cite[Theorem 4.5]{Giusti2003}, we know that $E$ is l.s.c. along sequences $u_n$ such that $\nabla^k u_n$ converges locally weakly in $L^p(\Omega)$ while $u$, $\nabla u, \ldots$, $\nabla^{k-1} u$ converge locally strongly in $L^p(\Omega)$. 

\medskip

\emph{First step: a preliminary observation}.
Let us start with a preliminary observation, by looking at what we can say about a sequence $(u_n)_{n \geq 1}$ such that 
\begin{equation*}
\sup_n \; E(u_n) < + \infty. 
\end{equation*}
Without loss of generality, in the sequel we assume that $\Omega$ is bounded, connected and has a smooth boundary: indeed, as \cite[Theorem 4.5]{Giusti2003} requires only \emph{local} convergence, we can always decompose $\Omega$ in a countable union of subdomain which are bounded, connected and with smooth boundaries, and patch pieces via a diagonal argument. 

By the growth assumption on $L$, we know that $\nabla^k u$ is bounded in $L^p(\Omega)$. We write $u_n = \tilde{u}_n + P_n$ where $(P_n)_{n \in \N}$ is a sequence of polynomials in $\Omega$ of degrees at most $k-1$ which match the moments up to order $k-1$ of $u_n$. By Sobolev embeddings we know that up to extraction the sequence $(\tilde{u}_n)_{n \geq 1}$ converges strongly in $W^{k-1,p}$ and weakly in $W^{k,p}$ to some $\tilde{u} \in W^{k,p}$.

\medskip

\emph{Second step: lower semi-continuity of $E$}. Armed with this preliminary observation, let us prove lower semi-continuity of $E$ for the convergence in $L^0(\Omega, \R^q;m)$. 
Indeed, following \cite[Theorem 4.5]{Giusti2003}, it is enough to prove that the polynomial $P_n$ converges to a limit. But $P_n$ matches the moment up to order $k-1$ of $u_n$, and the latter converges in $L^0(\Omega,\R^p;m)$ while $\Omega$ is bounded: it implies that the moments of $u_n$ converge to the ones of $u$, thus implying convergence of $(P_n)_{n \in \N}$ to a limit polynomial.

\medskip

\emph{Third step: compactness of the sublevel sets.} Then, let us take $\psi : \R^q \to [0, + \infty)$ with compact sublevel sets, and $(u_n)_{n \in \N}$ a sequence in $L^0(\Omega, \R^q;m)$ such that 
\begin{equation*}
\sup_n \left\{ E(u_n) + \int_\Omega \psi(u_n(x)) \, \ddr x \right\} < + \infty.
\end{equation*}
Without loss of generality we can assume that $\psi$ is radial and converges to $+ \infty$ at infinity. We want to prove that, up to extraction, $(u_n)_{n \in \N}$ converges to a limit in $L^0(\Omega,\R^q;m)$. By the preliminary observation we can write $u_n = \tilde{u}_n + P_n$, with $P_n$ polynomial of degree $k-1$ and $\tilde{u}_n$ a sequence which converges strongly in $W^{k-1,p}$. We only need to prove that, up to extraction, the sequence of polynomials $P_n$ converge, and for this we will use the function $\psi$. Assume by contradiction it is not the case. As $(P_n)_{n \geq 1}$ belongs to a space of finite dimension, it means that any norm of $P_n$ should diverge as $n \to + \infty$, for instance the $L^p(\Omega)$ norm. In particular, writing
\begin{equation*}
u_n = \tilde{u}_n + P_n = \| P_n \|_{L^p} \left( \frac{\tilde{u}_n}{\| P_n \|_{L^p}} + \frac{P_n}{\| P_n \|_{L^p}} \right),
\end{equation*} 
we know that $\tilde{u}_n / \| P_n \|_{L^p}$ converges to zero in $L^p(\Omega, \R^q;m)$, while the second term $\frac{P_n}{\| P_n \|_{L^p}}$ converges also in $L^p(\Omega, \R^q;m)$, up to extraction, to a polynomial $P$ of degree at most $k-1$ which is non zero as it belongs to the unit sphere of $L^p(\Omega, \R^q;m)$. In particular, with the help of Markov's inequality, again up to extraction of a subsequence, we can find a set $A$ of positive measure such that, for $n$ large enough and every $x \in A$,
\begin{equation*}
\left| \frac{\tilde{u}_n(x)}{\| P_n \|_{L^p}} + \frac{P_n(x)}{\| P_n \|_{L^p}} \right| \geq \frac{1}{2} |P(x)| \geq c
\end{equation*}
for some constant $c > 0$ independent on $n$.
We conclude by noticing that 
\begin{equation*}
\int_\Omega \psi(u_n(x)) \, \ddr x \geq \int_A \psi(u_n(x)) \, \ddr x = \int_A \psi\left( \| P_n \|_{L^p} \left( \frac{\tilde{u}_n}{\| P_n \|_{L^p}} + \frac{P_n}{\| P_n \|_{L^p}} \right) \right)   \geq m(A) \psi(c \| P_n \|_{L^p} ),
\end{equation*}
and that the right hand side goes to $+ \infty$ as $n \to + \infty$ which clearly contradict our assumption. 
\end{proof}

For the sake of completeness, and as we use it in the case of variational models with TV regularization, we also present the result in the case $p=1$. We only sketch the arguments of the proof as they are very similar to the one of Proposition~\ref{prop:checking_assumption_A}.

\begin{prop}
\label{prop:checking_assumption_A_p_=1}
We fix $X = \Omega$ a bounded open set of $\R^d$ that we endow with $m$ the Lebesgue measure. On the other hand the space $Y$ is taken to be $Y = \R^q$. We take $E : L^0(\Omega, \R^q;m) \to [0, + \infty]$ defined as in Definition~\ref{defi:E_from_W}, where $W : \R^{qd} \to [0, + \infty]$ satisfies Assumption~\ref{asmp:W} with $p=1$. 

Then the functional $E$ satisfies Assumption~\ref{asmp:reg_XYm_E}. 
\end{prop}

\begin{proof}
Using \cite[Theorem 2.34]{AFP} we know that $E$ is l.s.c. along sequences $u_n$ such that $\nabla u_n$ converges locally weakly as a measure.

Reasoning as in the proof of Proposition~\ref{prop:checking_assumption_A}, we can always restrict to the case where $\Omega$ is bounded, connected and has a smooth boundary. Similarly for any sequence $(u_n)_{n \geq 1}$  which satisfies $\sup_n E(u_n) < + \infty$ by embedding results \cite[Corollary 3.49]{AFP} we can write $u_n =  \tilde{u}_n + y_n$ where $(y_n)_{n \geq 1}$ is the sequence of the means of $\tilde{u_n}$ and $(\tilde{u}_n)_{n \geq 1}$ a sequence such that $(\tilde{u}_n)_{n \geq 1}$ converges strongly in $L^1$ while $(\nabla \tilde{u}_n)_{n \geq 1}$ converges weakly-$\star$ as a measure.

The rest of the proof follows \emph{mutatis mutandis} the second and third step of the proof of Proposition~\ref{prop:checking_assumption_A}, by replacing $P_n$ by $y_n$.
\end{proof}

\subsection{Functional of measures}
\label{sec:appendix_BW}

In this appendix, we state some standard properties of the functional $\B_W$ introduced in Definition~\ref{def:energyBW}.

First, note that with Assumption~\ref{asmp:W} it is easy to derive the coercivity estimate
\begin{equation}
\label{eq:coercivity_BW}
\B_W(\mubf,J) \geq C_1 |J|(\Omega \times \R^{q})^{qd} - C_2 |\mubf|(\Omega \times \R^q) = C_1 |J|(\Omega \times \R^{q})^{qd} - C_2 m(\Omega).
\end{equation}

The main result of this appendix is the ``dual'' representation of the functional.

\begin{prop}
\label{prop:dual_BW}
The function $(\mubf, J) \mapsto \B_W(\mubf,J)$ is jointly l.s.c. for weak-$\star$ convergence and convex. Moreover, being $W^*$ the Legendre transform of $W$ there holds
\begin{equation}
\label{eq:dual_BW}
\B_W(\mubf, J) = \sup_{a,b} \left\{ \int_{\Omega \times \R^q} a \, \ddr \mubf + \int_{\Omega \times \R^q} b \cdot \, \ddr J \ : \ a+W^*(b) \leq 0 \right\}
\end{equation}
where the supremum is taken over all pairs $(a,b)$ of continuous compactly supported functions from $\Omega \times \R^q$ to $\R$ and $\R^{qd}$ respectively.
\end{prop}

\begin{proof}
The dual representation formula~\eqref{eq:dual_BW}, at least on compact subsets of $\Omega \times \R^q$, can be for instance found in~\cite[Lemma 2.9]{chizat2018unbalanced}, where the authors quote Theorem 6 of~\cite{rockafellar1971integrals} together with Lemma A.2 of~\cite{bouchitte1988integral} for the proof. Extension to the non-compact set $\Omega \times \R^q$ is immediate by inner regularity of the integral. The lower semi-continuity, which is a consequence of~\eqref{eq:dual_BW}, can be also be found in~\cite[Theorem 2.34]{AFP}. 
\end{proof}

Note also that as soon as $\mubf$ is non-negative it is always optimal to choose $a = - W^*(b)$ if $W^*(b) < + \infty$, thus the formula reads
\begin{equation}
\label{eq:dual_BW_only_J}
\B_W(\mubf, J) = \sup_{b} \left\{ \int_{\Omega \times \R^q} b \cdot \, \ddr J - \int_{\Omega \times \R^q} W^*(b) \, \ddr \mubf  \right\}
\end{equation}
where the supremum is taken among all $b \in C_c(\Omega \times \R^q, \R^{qd})$ continuous, compactly supported and such that $W^*(b) < + \infty$. It yields directly that the Legendre transform of the map $J \mapsto \B_W(\mubf, J)$ is $b \mapsto \int W^*(b) \, \ddr \mubf$, restricted to $b$ such that $W^*(b) < + \infty$ everywhere.

\subsection{Regularizing maps of measures}

In this appendix we have also included the following lemma to regularize maps of measures. The arguments are by now classical in the case of curves of measures and can be found in~\cite[Chapter 8]{AGS} or~\cite[Chapter 5]{SantambrogioOTAM}. As we use it in Section~\ref{sec:difference_liftings} we state it here for the localized version of $\TEE$, see Section~\ref{sec:localization} for the notations. 

\begin{lm}
\label{lm:regularization_measure_valued}
Let $\mubf \in \Mm(A \times \R^q)$. If $\varepsilon > 0$, $\tilde{A}$ is an open set compactly embedded in $A$ and $\mathcal{V}$ is a neighborhood of $\mubf$ in $\Mm(A \times \R^q)$ for the topology of narrow convergence, then there exists $\tilde{\mubf} \in \mathcal{V}$ and $\tilde{v} : \tilde{A} \times \R^q \to \R^{qd}$ such that:
\begin{enumerate}
\item \emph{Absolute continuity}: the measure $\tilde{\mubf}$ has a density with respect to $\ddr x \otimes \ddr y$ the Lebesgue measure on $\tilde{A} \times \R^q$ that we denote by $\tilde{\rho}$. Moreover, for each $x$, the measure $\tilde{\rho}(x,y) \ddr y$ is a probability measure on $\R^q$.
\item \emph{Compact support in $y$}: there exists $B$ a ball of $\R^q$ such that $\tilde{\rho}$ is supported on $\tilde{A} \times B$.
\item \emph{Smoothness of $\tilde{\rho}$}: the density $\tilde{\rho}$ can be extended to a function of class $C^\infty$ on $\overline{\tilde{A} \times B}$ and bounded from below by a strictly positive constant.
\item \emph{Smoothness of $\tilde{v}$}: the velocity field $\tilde{v}$ is of class $C^\infty$ on $\overline{\tilde{A} \times B}$.
\item \emph{The velocity is almost tangent}: the equation $\nabla_x \tilde{\mubf} + \nabla_y \cdot (\tilde{v} \tilde{\mubf}) = 0$ is satisfied in the sense of distribution over $\tilde{A} \times \R^q$, so that, being $\mathbf{n}_B$ the outward normal to $B$, 
\begin{equation}
\label{eq:continuity_smooth}
\begin{cases}
\nabla_x \tilde{\rho} + \nabla_y \cdot (\tilde{v} \tilde{\rho}) = 0 & \text{in } \tilde{A} \times B, \\
\tilde{v} \cdot \mathbf{n}_B = 0 & \text{on } \tilde{A} \times \dr B, 
\end{cases}
\end{equation}
holds in the classical sense and moreover there holds
\begin{equation}
\label{eq:lm_reg_estimate_energy}
\B_W(\tilde{\mubf},\tilde{v} \tilde{\mubf},\tilde{A}) = \int_{\tilde{A} \times \R^q} W \left( \tilde{v}(x,y) \right)  \tilde{\rho}(x,y) \, \ddr x \ddr y \leq   \TEE(\mubf,A) + \varepsilon.
\end{equation}
\end{enumerate} 
\end{lm}

\begin{proof}
Let $J \in \M(A \times \R^q)^{qd}$ tangent to $\mubf$, that is, $\nabla_x \mubf + \nabla_y \cdot J = 0$ is satisfied in the sense of distributions and $\TEE(\mubf,A) = \B_W(\mubf,J,A)$. We will apply successive transformations to $\mubf$ and $J$ in order to obtain the regularity properties while staying in the neighborhood $\mathcal{V}$ of $\mubf$, and keeping the estimate~\eqref{eq:lm_reg_estimate_energy} valid. 

\medskip

\emph{First step: retraction on a ball in the $y$ variable}. For simplicity we assume $J \ll \mubf$. In the case $p=1$ of Assumption~\ref{asmp:W_stronger}, we can always fall back to this case by doing a small mollification of the $(\mu, J)$ variables, see the second step. Calling $B_R$ the centered ball of radius $R$, let us consider $F : \R^q \to B_R$ a smooth injective contraction onto $B_R$, the centered ball of radius $R$, which is the identity over $B_{R/2}$. Extending $F$ to $A \times \R^q$ by $\tilde{F}(x,y) = (x, F(y))$, we consider 
\begin{equation*}
\mubf^{(1)} = \tilde{F} \# \mubf, \qquad J^{(1)} =  \tilde{F} \# ( \nabla_y \tilde{F} \, J ).
\end{equation*} 
A simple change of variable yields $\nabla_x \mubf^{(1)} + \nabla_y \cdot J^{(1)} = 0$. Moreover, calling $v$ the density of $J$ with respect to $\mubf$, there holds
\begin{equation*}
\frac{\ddr J^{(1)}}{\ddr \mubf^{(1)}}(x,y) = (\nabla_y F \circ F^{-1})(y) v(x,y).
\end{equation*}
Assumption~\ref{asmp:W_stronger}, and in particular that $W$ is approximately radial, together with $\nabla_y F$ being non-expansive, yields
\begin{equation*}
W( (\nabla_y F \circ F^{-1})(y) v) \leq C_5 W(v) + C_6
\end{equation*}
for any $y \in \R^q$ and any $v \in \R^{qd}$ for some non-negative constants $C_5, C_6$. Thus we estimate as follows, taking in account that $F$ is the identity over $B_{R/2}$, so that $(\mubf, J)$ and $(\mubf^{(1)}, J^{(1)})$ coincide on $A \times B_{R/2}$:
\begin{align*}
\B_W(\mubf^{(1)},J^{(1)},A) & = \int_{A \times \R^q} W \left( \frac{\ddr J^{(1)}}{\ddr \mubf^{(1)}} \right) \, \ddr \mubf^{(1)} \\
& = \int_{A \times B_{R/2}} W \left( \frac{\ddr J^{(1)}}{\ddr \mubf^{(1)}} \right) \, \ddr \mubf^{(1)} + \int_{A \times (\R^q \setminus B_{R/2})} W \left( \frac{\ddr J^{(1)}}{\ddr \mubf^{(1)}} \right) \, \ddr \mubf^{(1)} \\
& \leq \int_{A \times B_{R/2}} W \left( \frac{\ddr J}{\ddr \mubf} \right) \, \ddr \mubf + \int_{A \times (\R^q \setminus B_{R/2})} \left( C_5 W \left( \frac{\ddr J}{\ddr \mubf} \right)  + C_6 \right)  \, \ddr \mubf \\
& \leq \B_W(\mubf, J, A) + \int_{A \times (\R^q \setminus B_{R/2})} \left( C_5 W \left( \frac{\ddr J}{\ddr \mubf} \right)  + C_6 \right)  \, \ddr \mubf.
\end{align*} 
As $R \to + \infty$ the second term converges to zero thanks to $\B_W(\mubf,J,A) < + \infty$, thus can be made arbitrary small: by appropriately choosing $R$ we can keep the validity of~\eqref{eq:lm_reg_estimate_energy}. Moreover as $R \to + \infty$ the measure $\mubf^{(1)}$ clearly converges to $\mubf$.

\medskip

\emph{Second step: smoothing by convolution}. We take a sequence $(\chi_n)_{n \in \N}$ of smooth and compactly supported mollifiers over $\R^{d} \times \R^q$ which form an approximation of unity, with support shrinking to $0$ as $n \to + \infty$. For $n$ large enough, we set 
\begin{equation*}
(\mubf^{(2)}, J^{(2)}) = (\chi_n * \mubf, \chi_n * J).
\end{equation*}
When restricted to $\tilde{A} \times \R^q$, the first marginal of $\mubf^{(2)}$ is the Lebesgue measure, and the support is included in $\tilde{A} \times B_{R+1}$ at least for $n$ large enough. Moreover as $\nabla_x \mubf^{(1)} + \nabla_y \cdot J^{(1)} = 0$ and that convolution commutes with differential operators we get $\nabla_x \mubf^{(2)} + \nabla_y \cdot J^{(2)} = 0$ on $\tilde{A} \times \R^q$. The measures $\mubf^{(2)}$ and $J^{(2)}$ have a density with respect to $\ddr x \otimes \ddr y$ and this density is smooth provided $\chi_n$ is smooth. We call $\tilde{\rho}^{(2)}$ the density of $\tilde{\mubf}^{(2)}$. Eventually, it is standard that convolution decreases the energy $\B_W$ (see e.g. \cite[Lemma 1.1.23]{chizat2017unbalanced}), thus:
\begin{equation*}
\B_W(\mubf^{(2)}, J^{(2)}, \tilde{A}) \leq \B_W(\mubf^{(1)}, J^{(1)}, A).
\end{equation*}
Of course by choosing $n$ large enough the restriction of $\mubf^{(2)}$ to $\tilde{A} \times \R^q$ will be close to $\mubf^{(1)}$ for the topology of narrow convergence.

\medskip

\emph{Third step: adding a positive constant}. Let $c > 0$ the constant such that $c \1_{y \in B_{R+1}} \, \ddr y$ is a probability measure. We set, for a parameter $\lambda \in (0,1]$,
\begin{equation*}
\tilde{\rho}^{(3)}(x,y) = (1-\lambda) \tilde{\rho}^{(2)}(x,y) + \lambda c \1_{y \in B_{R+1}} , \qquad J^{(3)} = (1-\lambda) J^{(2)}, 
\end{equation*} 
and we write $\mubf^{(3)}$ for the measure with density $\rho^{(3)}$ with respect to the Lebesgue measure. 
Note that $(\mubf^{(3)}, J^{(3)})$ is a convex combination of $(\mubf^{(2)}, J^{(2)})$ and $(c \1_{y \in B_{R+1}} \ddr x \ddr y,0)$. As both satisfy the equation $\nabla_x \mubf + \nabla_y \cdot J = 0$, so does their convex combination. Moreover, again by convexity of $\B_W$, we have
\begin{equation*}
\B_W(\mubf^{(3)}, J^{(3)}, \tilde{A}) \leq (1-\lambda) \B_W(\mubf^{(2)}, J^{(2)}, \tilde{A}) + \lambda m(\tilde{A}) W(0).
\end{equation*}
By taking $\lambda > 0$ small enough, we still satisfy~\eqref{eq:lm_reg_estimate_energy}, and by sending $\lambda \to 0$ we have that the measure $\mubf^{(3)}$ converges to $\mubf^{(2)}$. Note that the density $\tilde{\rho}^{(3)}$ is bounded uniformly from below by $c \lambda > 0$ on $\tilde{A} \times B_{R + 1}$.

\medskip

\emph{Fourth step: conclusion}. We take $\tilde{\mubf} = \mubf^{(3)}$ which has density $\tilde{\rho} = \tilde{\rho}^{(3)}$ with respect to the Lebesgue measure. The bounded support, smoothness and positivity of $\tilde{\mubf}$ comes respectively from the first, second and third point. We define $\tilde{v}$ as the density of $J^{(3)}$ with respect to $\mubf^{(3)}$. The smoothness of $\tilde{v}$ comes from the smoothness of $J^{(3)}$ and $\rho^{(3)}$, as well as the strict positivity of $\rho^{(3)}$. Equation~\eqref{eq:continuity_smooth} holds because $\nabla_x \mubf^{(3)} + \nabla_y \cdot J^{(3)} = 0$ holds in the sense of distributions. The estimate~\eqref{eq:lm_reg_estimate_energy} also holds as we made sure of it during the first, second and third step. If $R$ is large enough, $n$ is large enough and $\lambda > 0$ small enough then $\tilde{\mubf}$ indeed belongs to $\mathcal{V}$.   
\end{proof}

\bibliographystyle{abbrv}
\bibliography{bibliography}

\bigskip

  \textsc{Department of Decision Sciences, Bocconi University} \par
  Via Roberto Sarfatti, 25, 25100, Milan MI Italy.  \par 
  \textit{Email address}: \texttt{hugo.lavenant@unibocconi.it} \par

\end{document}